\documentclass[11pt,reqno,english]{amsart}

\usepackage[left=3cm,top=3cm,right=3cm,bottom=3cm]{geometry}

\usepackage{amsfonts,amsmath,latexsym,verbatim,amscd,mathrsfs,color,array}
\usepackage{amsmath,amssymb,amsthm,amsfonts,graphicx,color}
\usepackage{amssymb}
\usepackage{pdfsync}
\usepackage{epstopdf}
\usepackage[english]{babel}
\usepackage[latin1]{inputenc} 
\usepackage{amsmath}
\usepackage{amssymb}
\usepackage{amsfonts}
\usepackage{amsthm}
\usepackage{latexsym}
\usepackage{bbold}
\numberwithin{equation}{section}
\usepackage{color}      
\usepackage{epsfig}
\usepackage{cmap}

\newcommand{\vv}{\upsilon}
\newcommand{\M}{\mathcal{M}}
\newcommand{\ts}{\thicksim}
\newcommand{\R}{\mathbb{R}}

\newcommand{\NN}{\mathcal{N}}

\newcommand{\OO}{\mathcal{O}}

\newcommand{\pp}{\partial}

\newcommand{\y}{\textrm{y}}

\newcommand{\Q}{\textrm{Q}}
\newcommand{\N}{{\bf N}}

\newcommand{\T}{\mathcal{T}}

\newcommand{\A}{\alpha}
\newcommand{\D}{\delta}

\newcommand{\la}{\lambda}

\newtheorem{lemma}{Lemma}[section]
\newtheorem{prop}{Proposition}[section]
\newtheorem{theo}{Theorem}

\title{catenoidal layers for the Allen-Cahn equation in bounded domains}

\author[O. Agudelo]{Oscar Agudelo}
\address{\noindent O. Agudelo - Department of Mathematics,
Zapadoescka Univerzita v Plzni}
\email{oiagudel@ntis.zcu.cz}

\author[M. del Pino]{Manuel del Pino}
\address{\noindent M. Del Pino - Departamento de Ingenier\'{\i}a Matem\'atica and
CMM, Universidad de Chile, Casilla 170 Correo 3, Santiago, Chile.}
\email{delpino@dim.uchile.cl}

\author[J. Wei]{Juncheng Wei}
\address{\noindent J. Wei - Department of Mathematics, University of British Columbia, Vancouver BC V6T 1Z2, Canada }

\email{jcwei@math.ubc.ca}

\begin{document}

\maketitle

\begin{center}
\large{ Dedicated to Professor Haim Brezis on the occasion of his 70$^{th}$ birthday, with deep admiration}

\end{center}

\begin{abstract}
In this paper we present a new family of solutions to the singularly perturbed Allen-Cahn equation
$\A^2 \Delta u + u(1-u^2)=0, \quad \hbox{in }\Omega\subset \R^N
$
where $N=3$, $\Omega$ is a smooth bounded domain and $\A>0$ is a small parameter. We provide asymptotic behavior which shows that, as $\A\to0$, the level sets of the solutions collapse onto a bounded portion of a complete embedded minimal surface with finite total curvature that intersects orthogonally $\pp \Omega$ of the domain and that is non-degenerate respect to $\Omega$. We provide  explicit examples of surfaces to which our result applies.
\end{abstract}


\section{Introduction}
\subsection{Preliminary discussion}
In this paper we study the singulary perturbed boundary value problem
\begin{equation}\label{SPAC}
\A^2\Delta u + u(1-u^2)=0, \quad \hbox{in } \Omega,\quad  \quad \frac{\pp u}{\pp n}=0, \quad \hbox{on }\pp \Omega.
\end{equation}
where $\A >0$ is a small parameter, $\Omega\subset \R^N$ is a smooth bounded domain and $n$ is the inward unit normal vector to $\pp \Omega$.

\medskip
Solutions to \eqref{SPAC} correspond exactly to critical points of the Allen-Cahn energy
$$
J_{\A}(u):= \int_{\Omega}\frac{\A}{2}|\nabla u|^2 + \frac{1}{4\A}(1-u^2)^2, \quad u \in H^1(\Omega).
$$

\medskip
Equation \eqref{SPAC} arises for instance in the gradient theory of phase transition when modelling the phase of a material placed in $\Omega$ or when studying stationary solutions for bistable reaction kinetics, see \cite{29}.

Observe that $u=\pm 1$ are global minimizers of $J_{\A}$, representing  stable phases of two different materials isolated in $\Omega$.

\medskip
We are interested in solutions in solutions $u$ connecting the stable phases $\pm 1$. As described in \cite{36}, solutions of this type are expected to have narrow transtition layer from $-1$ to $+1$ with a nodal set that is asymptotically locally stationary for the perimeter functional. To be more precise in \cite{36}, the author showed that a family of local minimizers $\{u_{\A}\}_{\A}$ of $J_{\A}$ with uniformly bounded energy must converge in $L^1(\Omega)$, up to a subsequence,  to a function $u^*$, where
$$
u^*=\chi_{\Lambda} - \chi_{\Omega-\Lambda}
$$
where $\chi_{E}$ is the characteristic function of a set $E$ and $\Lambda \subset \Omega$ minimizes perimeter. In this case, as $\A \to 0$
\begin{equation}\label{ACENERGY}
J_{\A}(u_{\A})\to Per(\Lambda)\,\left(\int_{\R}\frac{1}{2}|w'|^2 +\frac{1}{4}(1-w^2)^2dt\right)
\end{equation}
where $w$ is the solution of
\begin{equation}\label{heteroclinic}
w'' + w(1-w^2)=0,\quad \hbox{in } \R, \quad \quad w'>0, \quad \quad w(\pm \infty)=\pm 1
\end{equation}
which is given by the explicit function
$$
w(t)=\tanh\left(\frac{t}{\sqrt{2}}\right), \quad t\in \R.
$$

Roughly speaking, the above assertion means that the level sets of $u_{\A}$ converge to $\pp \Lambda$ as $\A \to 0$. This result provided the intuition that ultimately lead to important developments in the theory of $\Gamma-$convergence and put into light a deep connection between the Allen-Cahn equation and the theory of minimal surfaces.  We refer reader to \cite{strongConvergenceCaffarelli1995,32, hutchinson_tonegawa00,33, 34} for related results and stronger notions of convergence.

\medskip
The connection between the Allen-Cahn equation and the theory of minimal surfaces has been explored in order to produce nontrivial solutions of \eqref{SPAC}, but the general understanding of solutions to these equation is far from being complete. In this regard
it is natural to ask for existence and asymptotic behavior of solutions to \eqref{SPAC} in the general setting. For the case of minimi\-zers we refer the reader to \cite{BardiPerthame, GarzaHumePadilla, PadillaTonegawa,SternbergZumbrun} and references there in. We also remark that local minimizers in convex domains are the constants $\pm 1$, see \cite{Matano, CastenHolland}.

\medskip
In low dimensions, for instance, $N=2$ Kohn and Sternberg in \cite{KohnSternberg}, using a measure theoretical approach  and taking advantage of the aforementioned intuition, constructed local minimizers $u_{\A}$ to \eqref{SPAC} with interfaces collapsing onto a fixed minimizing segment $\Gamma_0$ inside $\Omega$ that cuts $\pp \Omega$ perpendicularly and satisfies \eqref{ACENERGY}.

\medskip
In \cite{Kowalczyk} a situation similar to that one described in \cite{KohnSternberg} is consider but lifting the minimizing assumption on $\Gamma_{0}$, for nondegeneracy of this segment respect to the domain. Nodegeneracy of $\Gamma_0$ in this case is stated as
$$
K_0 + K_1 -|\Gamma_0|K_0K_1 \neq 0
$$
where $K_0,K_1$ are the curvatures of $\pp\Omega$ at the points where $\Gamma_0$ cuts $\pp \Omega$ orthogonally.

This geometrical condition is equivalent to the fact that the eigenvalue problem
$$
h'' = \la h, \quad \hbox{in }(0,l), \quad \quad
K_0h(0)+h'(0)=0, \quad K_1h(l)-h'(l)=0, \quad l=|\Gamma_0|
$$
does not not have $\la=0$ as an eigenvalue. The author also provides a full description of the solutions which can have Morse Index either one or two depending on the sign of $K_0$ and $K_1$.

\medskip
Later this construction was generalized in \cite{delPinoKowalczykWei} under the same geometrical setting described in \cite{Kowalczyk}, but for multiple transitions that in the limit collapse onto the segment $\Gamma_0$. There the transition layers at main order interact exponentially respect to their mutual distances giving rise to Toda system of odes.

\medskip
In dimension $N=3$, Sakamoto in \cite{Sakamoto} constructed solutions to \eqref{SPAC} having a narrow transition through a planar disk that cuts orthogonally the boundary of the domain and is nondegenerate in a suitable sense. The author also provides a characterization for this nondegeneracy in terms of spectrum of the dirichlet to neumann map of the planar disk. As for higher dimension in the setting of manifolds, Pacard and Ritore in \cite{PacardRitore} constructed solutions having a narrow transition along a codimension one nondegenerate minimal submanifold.

\medskip
In the the spirit of the results mentioned above, we also want to refer the reader to \cite{37, DuGui10,DuWang12,DuLai} dealing with similar results for the inhomogeneous allen-cahn equation and \cite{JunYangJunchengWeiI, JunYangJunchengWeiII,GuoJunYang} for a semilinear elliptic problems where resonance phenomena appear.

\medskip
The underlying geometric problem is the existence of minimal surfaces inside the domain $\Omega$ intersecting the boundary $\partial \Omega$ orthogonally. This problem, in a general three dimensional compact Riemannian manifold, has been completely settled in a recent paper by M. Li \cite{Li}. For earlier results in this direction we refer to Fraser-Li \cite{Fraser-Li} and Fraser-Schoen \cite{FraserSchoen}. The uniqueness  of the critical catenoid in a ball is proved by Fraser-Schoen \cite{FraserSchoen2}.

\subsection{Main result}
Our goal in this paper is to generalize the results in \cite{Kowalczyk,Sakamoto} respectively by taking $N=3$, and a more general class of minimal surfaces for limiting nodal set.

   \medskip
In order to state our main result, let $M$ be a complete embedded minimal surface of finite total curvature in $\R^3$. For over a century there were only two known examples of such surfaces, namely the plane and the catenoid. In \cite{CostaI,CostaII} Costa gave the first nontrivial example of such surface with genus one, being properly embedded and having two catenoidal connected components outside a large ball sharing an axis of symmetry and another planar component perpendicular to this axis. Later this construction was generalized in \cite{HoffmanMeeksI,HoffmanMeeksII} to surfaces having the same look as the Costa's surface far away but with arbitrary genus. We refer the interested reader to \cite{Kapouleas,PerezRos} and references there in, for related results and further  generalizations.

\medskip
It is known that $M$ is orientable and $\R^3 - M$ has exactly two connected components namely $S^+$ and $S^-$, see \cite{HoffmanKarcher}. We set by $\nu: M \to \mathcal{S}^2$ the fixed choice of the unit normal vector of $M$ pointing towards $S^+$.

\medskip
For $x=(x_1,x_2,x_3)\in\R^3$ we denote
$$
r=r(x)=\sqrt{x^2_1+x^2_2}.
$$

It is also known that for some large but fixed $R_0>0$, outside the infinite cylinder $r>R_0$, $M$ decomposes into finite connected components, say $M_1,\ldots,M_m$, which from now on we will refer to as the ends of $M$. For every $k=1,\ldots,k$, there exist a smooth function $F_k=F_k(y')$ such that
$$
M_k=\left\{(y',y_3)\in \R^3\,:\, r(y',y_3)>R_0 ,\quad y_3=F_k(y') \right\}
$$
where $F_k$ has the asymptotic expansion
\begin{equation}\label{asymptoticEnds}
F_k(y')= a_k\log(r) + b_k + b_{ik}\frac{y_i}{r^2} + \OO(r^{-3}), \quad \hbox{as }r\to \infty
\end{equation}
for some constants $a_k,b_k,b_{ik}$ satisfying
$$
a_1\leq a_2\leq \dots\leq a_k, \quad \sum_{k=1}^{m} a_k=0
$$
and relation \eqref{asymptoticEnds} can be differentiated. In this case we set the connected component $S^+$ of $\R^3-M$ to be the one containing the axis $x_3$ which corresponds to the axis of symmetry of the ends $M_1, \ldots,M_k$.

\medskip
We denote by $\nu:M\to S^2$ the unit normal vector to $M$ pointing towards $S^+$ and we take \emph{Fermi coordinates} near $M$
$$
x=y+z\nu(y),\quad y\in M, \quad |z|<\eta + \D\log(2+r(y))
$$
for some $\eta,\D>0$ small. From now on $\D$ will represent a small, but fixed number, independent of $\A>0$. Observe that also that $\R^3$ is spanned by the moving frame $T_{y}(M)\oplus\R_{\nu(y)}$ and that $z$ corresponds to the signed distance to $M$, i.e
$$
|z|=dist(x,M), \quad x=y+z\nu(y)
$$
for every $y\in M$ and $z$ small.

\medskip
Next, consider a smooth bounded domain $\Omega$ such that
\begin{itemize}
\item[(i).] $\Omega$ contains a portion of $M$ which we denote by $\M$.

\medskip
\item[(ii).]  $\Omega - \bar{\M}$ has two connected components  which, abusing the notation, we naturally denote as $S^+, S^-$ with the same convention as before.

\medskip
\item[(iii).] $\pp \Omega \cap M = \cup_{k=1}^m M_{k}\cap \pp \Omega$ where for every $k=1, \ldots ,m$,  $C_k:=M_k \cap \pp\Omega$ is a smooth closed simple curve.
\end{itemize}

Notice that the $\pp \Omega \cap \bar{\M}$ consists of nonintersectiong closed curves, since the curve $C_k \subset M_{k}$.

\medskip
Following \cite{KohnSternberg, Kowalczyk, delPinoKowalczykWei} in order to produce solutions of \eqref{SPAC} we need some kind of criticality and nondegeneracy assumptions on $\M$ respect to $\bar{\Omega}$. To do so, let us introduce $\Delta_{\M}$ the laplace-beltrami operator of $M$, and $|A_{\M}|$ the norm of the second fundamental form of $M$ and let us consider now the following eigenvalue problem
\begin{equation}\label{nondegM}
\Delta_{\M} h + |A_{\M}|^2 h = \la h, \quad \hbox{in }\M, \quad \frac{\pp h}{\pp \tau} + \kappa(y)h =0, \quad \hbox{on }\M
\end{equation}
where $\tau$ represents the inward unit normal direction to $\M$ and $\kappa(y)$ is given by
$$
\kappa(y):= \left<\frac{\pp n}{\pp \nu};\nu\right> ,\quad y\in \M.
$$
where we recall that $n$ is the unit normal vector to $\pp \Omega$. We remark that since $\M$ is a minimal surface it follows that
$$
|A_{\M}|^2=-2K_{\M}
$$
with $K_{\M}$ being the gaussian curvature of $\M$.

\medskip
Our crucial assumptions on $\M$ are the following.

\begin{itemize}
\item[(I).]$\M$ cuts orthogonally $\pp \Omega$ along the curves $C_k$, for $k=1,\ldots,m$.

\medskip
\item[(II).] The eigenvalue problem \eqref{nondegM} in $H^1(\M)$ does not have $\la =0$ as an eigenvalue.
\end{itemize}

\medskip
As stated in \cite{GuoJunYang}, from assumption (I) it follows that $\tau$ and $n$ must be parallel along every curve $C_k$ and  consequently these curves must be geodesics in $\pp \Omega$ in the direction of $\nu$ since their the normal vectors in $\pp \Omega$ are parallel to $n$. Therefore, the quantity $\kappa(y)$  corresponds to  the geodesic curvature of $\pp \Omega$ in the direction $\nu(y)$ for $y\in C_k$.

\medskip
Our main result is the following.

\begin{theo}\label{theo1}
Assume conditions (i)-(ii)-(iii) and (I)-(II). Then for every $\A>0$ small enough there exists a solution $u_{\A}$ to \eqref{SPAC} such that
$$
u_{\A}(x)= w\left(\frac{z - h(y)}{\A} \right) + \OO_{H^1(\Omega)}(\A),
$$
for $x=y+z\nu(y)\in \NN = \{dist(\cdot, M)<\eta\}\cap \Omega$, where the function $h$ solves at main order the boundary value problem
\begin{equation}\label{REDUCEDOPE}
\Delta_{\M} h + |A_{\M}|^2 h = 0, \quad \hbox{in }\M, \quad \frac{\pp h}{\pp \tau} + \kappa(y)h =\OO(\A) , \quad \hbox{on }\pp \M.
\end{equation}

While outside $\NN$,
\begin{equation}
u_{\A}(x)=\left\{
\begin{array}{ccc}
+1,& x\in S^+\\
\\
-1,&x\in S^{-}
\end{array}
\right.
\end{equation}
as $\A \to 0$.
\end{theo}



Theorem \ref{theo1} provides a  rather general description of an explicit family of minimal surfaces for which the construction applies and allows us to consider more involved examples of the one presented in \cite{Sakamoto}.
We also remark that if our original surface and the domain $\Omega$ have axial symmetry one can reduce our developments to this setting and condition (II) can be recasted as requiring that $\la=0$ is not an eigenvalue, for problem \eqref{nondegM} in $H^1_{axial}(\M)$.

\medskip
The paper is organized as follows. In section \ref{JacobiOprExampl} we present briefly invertibility theory of the operator descibed in \eqref{REDUCEDOPE} with robin boundary conditions and we also disscuss some example where our result apllies. Next in section \ref{GEOMCOMP}, we present the geometric framework we will use to set up the proof of Theorem \ref{theo1}. In section \ref{APPROXSOL} we construct an accurate approximation of the solution to our problem and then in section \ref{PROOFTHEO1} we sketch the proof of our main result. The final section is devoted to present detailed proofs of lemmas and propositions used in section \ref{PROOFTHEO1}.

\medskip
{\bf Acknowledgments: } The research of the first author was supported by the Grant 13-00863S of the Grant Agency of the Czech Republic. M. del Pino has been partly supported by
a Fondecyt grant and by Fondo Basal CMM.  J. Wei is partially supported by NSERC of Canada.


\medskip
\section{jacobi operator with Robin BOundary conditions}\label{JacobiOprExampl}

In this part we consider the equation
$$
\Delta_{\M} h + |A_{\M}|^2 h
= f, \quad \frac{\pp h}{\pp \tau}+I(y)h=0, \quad \hbox{on }\pp M
$$

Using Fourier decomposition and the nondegeneracy assumption (II), it is straight-forward to verify that for any $f \in L^2(\M)$, there exists a unique solution $h\in H^2(\M)$ satisfying
$$
\|h\|_{W^{2,2}(\M)}\leq C \|f\|_{L^2(\M)}.
$$

If $p>2$ and $f\in L^{p}(\M)$ then by standard regularity theory $h\in W^{2,p}(\M)\cap C^{1,1-\frac{2}{p}}(\M)$ with a priori estimate
$$
\|h\|_{*}= \|D^2 h\|_{L^p(\M)} +\|\nabla h\|_{L^{\infty}(\M)}+ \|h\|_{L^{\infty}(\M)} \leq C\|f\|_{L^p(\M)}.
$$

Directly from  this it follows that for any $f\in L^{p}(\M)$, $g\in L^{p}(\pp M)$
$$
\Delta_{\M} h + |A_{\M}|^2 h
= f,\quad  \hbox{in }\M, \quad \frac{\pp h}{\pp \tau}+I(y)h=g, \quad \hbox{on }\pp \M
$$
has a unique solution satisfying
$$
\|h\|_{*}\leq C\left(\|f\|_{L^p(\M)} + \|g\|_{L^p(\pp \M)}\right).
$$

\subsection{Examples}
In this part we discuss some particular situations where our theorem applies. Let $M$ be the catenoid in $\R^3$ parameterized by the mapping
$$
Y(\y,\theta):= \left(\sqrt{1+\y^2}\cos\theta,\sqrt{1+\y^2}\sin\theta, \log\left(\y + \sqrt{1+\y^2}\right)\right), \quad \y\in\R, \quad \theta \in (0,2\pi)
$$
which provides coordinates on $M$ in terms of the signed arch-length of the profie curve and the rotation around the $x_3$-axis, which in our setting corresponds to the axis of symmetry of $M$.

\medskip
The unit normal vector to $M$ pointing towards $S^+$, is given by
$$
\nu (\y,\theta) =   \frac 1{\sqrt{1+\y^2}} \, (- \cos\theta, -\sin\theta,  \y), \quad \y\in\R, \quad \theta \in (0,2\pi)
$$
and we consider the \emph{Fermi coordinates}
$$
\tilde{X}( \y,\theta, z) = Y(\y,\theta) + z \nu(\y,\theta)
$$
which define a change of variables for instance on the neighborhood of $M$,
$$
\NN= \left\{ Y(\y,\theta) + z \nu(\y,\theta) : |z| < \eta +\frac 12 \ln (1+ \y^2)
\right\}
$$
for some fixed and small $\eta >0$.

\medskip
Now we assume that $\Omega$ is axially symmetric. Since $M$ has two ends and due to the axial symmetry, $\pp \Omega\cap M = C_1 \cup C_2$ where $C_1,C_2$ are parallel, nonintersecting circles parameterized respectively by
$$
Y_i(\theta):=Y(\y_i,\theta), \quad \theta \in (0,2\pi), \quad i=1,2
$$
for some fixed $\y_1<\y_2$.

\medskip
To describe $\pp \Omega$ close to the circles $Y_i(\theta)$, we can assume the existence of two smooth functions
$$
G_1,G_2:(-\eta,\eta)\to \R, \quad  G_i(0):=\y_i, \quad i=1,2
$$
so that the two systems of coordinates
\begin{equation}\label{bdrcondit}
X_i(\theta,z):=Y(G_i(z),\theta) + z\nu (G_i(z),\theta), \quad \theta \in (0,2\pi), \quad |z|<\eta, \quad i=1,2
\end{equation}
describe the set
$$
\pp \Omega \cap \NN= \{x\in \pp \Omega: x=X_i(\theta,z),\quad |z|<\eta, \quad \theta \in (0,2\pi), \quad i=0,1\}.
$$


\medskip
In order to explain geometrically conditions (I), (II), first recall that $\M=\Omega \cap M$ and consider a function $h\in C^2(\M)$ such that $\|h\|_{C^2(\M)}< \eta$. A normal deformation of $\M$ within $\Omega$ can be described by the coordinate system
\begin{equation}\label{MODIFIEDFERMICOORD}
\tilde{Y}_{h}(y,\theta)\,:=\,Y\left(\y(y,\theta),\theta\right) + h(y,\theta)\,\nu\left(\y(y,\theta),\theta\right), \quad \y_1<y<\y_2, \quad \theta \in (0,2\pi)
\end{equation}where
\begin{eqnarray*}
\y(y,\theta)&:=& \frac{G_2(h(y,\theta))- G_2(h(y,\theta))}{\y_2-\y_1}(y-\y_1) + G_{1}(h(y,\theta)), \quad \y_1<y<\y_2, \quad \theta \in (0,2\pi)
\end{eqnarray*}

Denoting $\M_h:= \tilde{Y}_{h}([\y_1,\y_2]\times(0,2\pi))$ and $\det g_h$ its respective induced metric, with the convention that $g_0$ is the induced metric of $\M$. The area functional of $\M_h$ is computed as
\begin{equation}\label{AREAFUNCT}
\mathcal{A}(\M_h):=\int_{\M_h}1\,dA_{g_h}= \int_0^{2\pi} \int_{\y_0}^{{\y_1}} \sqrt{\det g_h}dy d\theta.
\end{equation}

This area functional is of class $C^2$ and its first variation around $\M$ is given by
\begin{equation}\label{criticalcondM}
D\mathcal{A}(\M)[h]=-\int_{C_1\cup C_2} \left(\frac{\pp_zG_2(0) -\pp_zG_1(0)}{\y_2 -\y_1}(y -\y_1) + \pp_{z}G_1(0) \right)h(y,\theta)ds_{g_0} + \int_{\M}H_{\M}\,h\,dA_{g_0}
\end{equation}
where $H_{\M}$ is the mean curvature of $\M$.

\medskip
We notice from \eqref{criticalcondM} that $\M$ is critical for the area functional \eqref{AREAFUNCT} if
\begin{equation}\label{criticalityofM}
H_{\M}=0, \quad \pp_{z} G_1(0)=\pp_{z}G_2(0)=0.
\end{equation}

Therefore condition (I) is equivalent to saying that $\M$ is critical for the functional \eqref{AREAFUNCT} respect to normal perturbation of $\M$.

\medskip
Since $\M$ is a minimal surface, condition $H_{\M}=0$ is automatically satisfied. Consequently, the criticality of $\M$ reduces to saying that $\M$ must cut $\pp \Omega$ orthogonally.

\medskip
Assuming condition (I), the second variation of the area functional around $\M$ is given by the quadratic form
$$
D^2\mathcal{A}(\M)[h,h]:= (-1)^{i+1}\int_{Ci} \, \,\pp_{zz}G_i(0)h^2(y)ds_{g_0}  + \int_{\M}\left(|\nabla_{\M} h|^2 - |A_{\M}|^2h \right)dA_{g_0}
$$
and stability properties of $\M$ respect to $\Omega$ are analyzed through the linear robin boundary value problem
\begin{equation}\label{JACOBIROBINBDVALPROB}
\Delta_{\M} h + |A_{\M}|^2 h =\la\, h, \quad \hbox{in }\M, \quad \frac{\pp h}{\pp \tau_i} +\pp_{zz}G_i(0)h=0, \quad \hbox{on }\pp \M.
\end{equation}

\medskip
We make the important remark that, under assumption (I), $n=\tau_i$ where $\tau_i$ is the inward unit tangent vector of the profile curve of the catenoid along the circle $C_i$.

\medskip
Also observe that $K_i:=(-1)^{i+1}\pp_{zz} G_i(0)$ corresponds to the curvature of the integral curve of $\pp \Omega$ in the direction of $\nu$ along the circle $C_i$.

\medskip
From the axial symmetry, nondegeneracy of $\M$ reduces asking that the only solution to the boundary value problem
$$
\pp_{\y\y} h + \frac{\y}{1+\y^2}\pp_{\y} h +\frac{2}{(1+\y^2)^2}h =0, \quad \y_1<\y<\y_2, \quad \frac{\pp h}{\pp \y}(\y_i) +(-1)^iK_ih(\y_i)=0.
$$
is the trivial one.

\medskip
The linear equation
\begin{equation}\label{JACOBIOPE}
\Delta_{M} h + |A_{M}|^2h =0, \quad \hbox{in } M
\end{equation}
has two axially symmetric entire solutions
$$
z_1(y)=y\cdot e_3 , \quad z_2(y)=y\cdot \nu(y),\quad  y \in M
$$
corresponding  respectively to the invariances of the entire catenoid $M$ under translations along the vertical axis and dilations. We refer the reader to section 4 in \cite{AgudelodelPinoWei}  and \cite{9} for full details.

\medskip
We directly check that in coordinates $y=Y(\y,\theta) \in M$
$$
z_1(\y)=\frac{\y}{\sqrt{1+\y^2}}, \quad z_2(\y)=\frac{\y}{\sqrt{1+\y^2}}\log(\y + \sqrt{1+\y^2}) -1, \quad \y\in  \R
$$
from where we observe that $z_1$ is odd and $z_2$ is even.

\medskip
We remark also that $z_1, z_2$ are strictly increasing in the variable $\y \in (0, \infty)$, since
$$
\pp_{\y}z_1(\y)= \frac{1}{(1+\y^2)^{\frac{3}{2}}} , \quad
\pp_{\y}z_2(\y)= \frac{\y}{1+ \y^2} + \frac{\log\left(\y + \sqrt{1 + \y^2}\right)}{(1+\y^2)^{\frac{3}{2}}}, \quad \y \in \R
$$
and also $z_2$ changes sign at only one point $\y=\y_0>0$ and $z_2(\y)<0$ for $-\y_0 < \y<\y_0$. While $z_1$ changes sign once at $\y=0$.

\medskip
Using basic theory of odes and following the developments from section 4 in \cite{AgudelodelPinoWei}, we find that $\la=0$ is not an eigenvalue of \eqref{JACOBIROBINBDVALPROB} if and only if
\begin{equation}\label{NONDEGDET}
\det\left[
\begin{array}{cc}
\pp_{\y}z_1(\y_1) + K_1 z_1(\y_1) & \pp_{\y}z_2(\y_1) + K_1 z_2(\y_1)\\
\pp_{\y}z_1(\y_2) - K_2 z_1(\y_2) & \pp_{\y}z_2(\y_2) - K_2 z_2(\y_2)\\
\end{array}
\right]\neq 0.
\end{equation}

\medskip
In particular, we directly check that in the case that $\pp \Omega$ is almost flat, i.e $K_1=K_2=0$, $\M$ is nondegenerate. We also remark that condition \eqref{NONDEGDET} is clearly invariant under dilations.

\medskip
Consider the case
$$
\Omega:=\left\{\frac{x_1^2}{a^2} + \frac{x_2^2}{a^2} + \frac{x_3^2}{b^2}=1\right\}
$$
an ellipsoid of revolution, where $a,b>0$ and $\M$ being and  even catenoidal portion inside $\Omega$.

\medskip
In this case from the disscussion above, we can describe $\pp \Omega$ near $\M$ as the level set
$$
\frac{1}{a^2}\left(\sqrt{1 + \y^2} - \frac{z}{\sqrt{1+\y^2}}\right)^2+\frac{1}{b^2}\left(\log(\y + \sqrt{1 + \y^2}) + z\,\frac{\y}{\sqrt{1+\y^2}}\right)^2 =1.
$$

Using implicit function theorem we find that
$$
\pp_{z}G_2(0)= -\pp_{z}G_1(0)= \frac{-\frac{1}{a^2} + \frac{1}{b^2}\frac{\bar{\y}}{\sqrt{1+\bar{\y}^2}}\log(\bar{\y} + \sqrt{1+\bar{\y}^2})}{\frac{\bar{\y}}{a^2} + \frac{1}{b^2}\frac{1}{\sqrt{1+\bar{\y}^2}}\log(\bar{\y} + \sqrt{1+\bar{\y}^2})}
$$
so that for $\M$ to be critical respect to the ellipsoid $\Omega$ we need to choose the point $\bar{\y}$ satisfying
$$
z_2(\bar{\y})=\frac{1}{a^2} - \frac{1}{b^2}
$$
and from the monotonicity of $z_2(\y)$ we see that once the ellipsoid has been fixed there is exactly one catenoid that cuts the boundary of the ellipsoid perperdicularly. This is our candidate of critical and nondegenerate minimal surface inside the ellipsoid.

\medskip
In order to compute the geodesic curvatures we apply again the implicit function theorem to find that
$$
K_2=K_1 = -\frac{1}{a}+\frac{1}{b^2}
$$

\medskip

In the case $\Omega=B_R(0)$, with $R=a=b$, we observe that $\M$ is the so called critical catenoid.  This setting appeared in \cite{FraserSchoen} as solution of an maximization problem for the first steklov eigenvalue of  the dirichlet to neumann mapping in bounded domains.

\medskip
Using the same notation as above, one can verify that
$K_1=K_2=\frac{\bar{\y}}{1+\bar{\y}^2}= \frac{1}{R}$ and the nondegeneracy condition translates into  saying that $\bar{\y}\neq \frac{1}{\sqrt{2}}$ which holds true since $\bar{\y}\sim 1.5$.

\medskip
Concerning stability issues let us consider the quadratic form
$$
\Q(h,h):= - \int_{\pp \M}k(y)h^2d s_{g_0} + \int_{\M} \left(|\nabla h|^2 - |A_{\M}|^2h^2 \right)dA_{g_0}, \quad h\in H^1(\M).
$$

\medskip
We first stablish conditions $\M$ to be minimizer of the area functional.

\begin{prop}\label{minimizingsituation}
Assume  $z$ is a smooth positive solution to the linear equation \eqref{JACOBIOPE} in an open set of $M$, containing $\M$. For every smooth function $\varphi$ in $\M$ it holds that
$$
Q(\varphi,\varphi)=\int_{\pp \M} \left(\frac{\pp log(z)}{\pp \tau} - k(y)\right)\varphi^2ds_{g_0} + \int_{\M}|\nabla \varphi - \frac{\varphi}{z}\nabla z|^2 dA_{g_0}
$$
where $\tau$ is the outer normal vector to $\pp \M$. Consequently if
$$
k(y) < \frac{\pp \log(z)}{\pp \tau}, \quad y\in \pp \M
$$
then $\M$ is minimizer for the area functional.
\end{prop}

\begin{proof}
The proof follows directly testing equation \eqref{JACOBIOPE} against $\psi=\frac{\varphi^2}{z}$ and integrating by parts.
\end{proof}

\medskip
We directly see from the previous proposition that if $K_1,K_0 <0$ and $\M$ is an even catenoidal portion with small area, we are in the case $\Omega$ nonconvex and $\M$ a local minimizer for the area functional. To see this it is enough to consider the value $\bar{\y}>0$ from above and $\M$ to be an even piece of catenoid contained in the portion associated to $\bar{\y}$ where $z=-z_2>0$.

\medskip
If $\M$ has large enough area it resembles the entire catenoid which has positive morse index. Cutting off an eigenfunction of $\Delta_{M} + |A_{M}|^2$ associated to a positive eigenvalue in a way that the boundary condition is not seen one can obtain a direction where the second variation of surface area is negative and therefore a catenoidal portion $\M$ with large area is unstable.

\medskip
The former situation ocurrs also in a general complete embedded minimal surface with finite total curvature for which the Morse Index is finite.

\section{Geometrical computations}\label{GEOMCOMP}

In this part we compute the euclidean laplacian in a neighborhood $\NN$ of $\M$ and the normal derivate $\frac{\pp}{\pp n}$ in $\pp\Omega \cap \NN$ in system of coordinates that will be suitable for our subsequent developments.

\medskip
In what follows we consider a large dilation of $\M$, denoted b $M_{\A}:= \A^{-1}\M$ for $\A >0$ small. First we compute the euclidean laplacian well inside the set $\Omega$ close to $\M$. Following the developments from \cite{9}, for the complete embedded nondegenerate minimal surface $M$ with finite total curvature denote, $M_{0}$ the part of $M$ inside the cylinder $\{r(\y)<R_{0}+1\}$. To parameterize $M_0$ we take a mapping $\y \in \mathcal{U}\subset \R^2\to y:=Y_{0}(\y)$ with associated induced metric given by
$g:=\left(g_{ij}\right)_{2\times 2}$. The laplace beltrami operator of $M$ inside the cylinder can be computed as
\begin{equation}\label{LAPBELTOPR}
\Delta_{M}=  \frac{1}{\sqrt{\det(g)}}\pp_{i}\left(\sqrt{\det(g)
}g^{ij}\,\,\pp_j\right)=a^{0}_{ij}\pp_{ij} + b^{0}_{i}\pp_i
\end{equation}
where $g^{-1}=(g^{ij}))_{2\times2}$ is the inverse of $g$ and where the coefficient matrix $(a^0_{ij})_{2\times2}$ is an smooth uniformly elliptic matrix.

\medskip
Next, consider the set
$$
D:=\left\{\y=(\y_1,\y_2) \in \R^2\,:\, r(\y)>R_0\right\}
$$
and for $M_k$, the $k-th$ end of $M$, we use the parametrization
$$
\y \in D \mapsto Y_{k}(\y):= \y_i e_i + F_{k}(\y)e_3.
$$

Next, we notice that the unit normal vector to $M$ at a point $y \in M_{k}$ has the expression in coordinates
\begin{eqnarray*}
\nu(y)&:=& \frac{(-1)^k}{\sqrt{1+|\nabla F_{k}(\y)|^2}}\left(\pp_{i}F_{k}e_{i}-e_{3}\right)\\
&=& (-1)^{k}e_{3} + a_{k}\frac{\y_i}{r^{2}} + \OO(r^{-2}), \quad y=Y_{k}(\y)
\end{eqnarray*}
so that $\pp_{i}\nu = \OO(r^{-2})$ and $|A_{M}|^2=\OO(r^{-4})$ as $r\to \infty$.
\medskip

In the coordinates $Y_{k}(\y)$ on $M_k$ the metric $g:=\left(g_{ij}\right)_{2\times 2}$ satistisfies
$$
g_{ij}= \D_{ij} + \OO(r^{-2}),\quad i,j=1,2 \quad \hbox{as } r \to \infty
$$
and this relations can be differentiated in the sense that the term $\OO(r^{-2})$ gains one negative power for every time it is differentiated.

\medskip
We compute the laplace-beltrami operator on $M_{k}$, using \eqref{LAPBELTOPR} to find
\begin{equation}\label{LAPBELTOPR2}
\Delta_{M} = \Delta_{\y} + \OO(r^{-2})\pp_{ij} + \OO(r^{-3})\pp_i , \quad \hbox{on }M_k.
\end{equation}

$M$ is parameterized completely by the $m+1$ local coordinates described above and we observe that expression \eqref{LAPBELTOPR} holds in the entire $M$, where for $k=1,\ldots,m$ the coefficients on $M_k$ satisfy
$$
a^0_{ij}(y)=\D_{ij} + \OO(r^{-2}),\quad  b_i(y)=\OO(r^{-3})\quad \hbox{for }k=1,\ldots,m, \quad \hbox{as } r\to \infty.
$$

\medskip

\medskip
We can consider Fermi coordinates given by the mapping $X(y,z):= y + z \nu(y)$ which provides a change of variables in the neighborhood of $M$
$$
\NN:=\left\{x=y+z\nu(y)\,:\, |z|< \eta + \D\log\left(2 + r(y)\right)\right\}
$$
and where we have the expression
\begin{equation}
\Delta_{X}  =   \partial_{zz} + \Delta_{M} \,-\, z|A_{M}|^2 \,\pp_z\,+\, D \label{lap3}
\end{equation}
where the $\Delta_{M}$ is given in \eqref{LAPBELTOPR}, \eqref{LAPBELTOPR2} and where the differential operator $D$ has the form
$$
D \,=\, z\, a^1_{ij}(y,z)\, \pp_{ij} \,+\, z \,b_i^1(y,z)
\,\pp_{i} \,+\, z^3
\,b_3^1(y,z)\,\pp_z
$$
and the smooth functions $a_{ij}(y,z) $, $b_{i}(y,z)$  satisfy on the ends of $M$
\begin{equation}\label{AsymptRestLap2}
\begin{array}{cc}
|a_{ij}^1| \,+\, |r\, \nabla a_{ij}^1| \,=\, O( r^{-2}), \quad |b_i^1| \,+\,
|r\, \nabla b_i^1| \,=\, O( r^{-3}) \\
\\
|b_3^1| \,+\, |r\, \nabla b_3^1| \,=\, O( r^{-8})
\end{array}
\end{equation}
as $r\to \infty$, uniformly on $z$ in the neighborhood $\NN$ of
$M$, see lemma 2.1 \cite{9}.

\medskip
The term $b_3^1$ comes from the mean curvature of the normally translated surface
$$
M_{z}:= \{y+z\nu(y)\,:\, y\in M\}\subset \NN
$$
for fxed $z$ small. It is well known that
$$
H_{M_z}:= H_{M} - z|A_{M}|^2 + \underbrace{z^2(k_1^3+k_2^3)+ z^3 (k_1^4 + k_2^4) + \OO(z^4\,r^{-10})}_{b_{3}^1(y,z)}.
$$
where $k_1,k_2$ are the principal curvatures of $M$. From the asymptotics of $\nabla_{M} \nu$, we have that $k_i = \OO(r^{-2})$ and since $H_{M}=k_1+k_2 =0$, it follows that $k_1^3+k_2^3 =0$ and therefore the expansion for $b_3^1$ in \eqref{AsymptRestLap2} follows.

\medskip
Recall that for small $\A>0$ $M_{\A}:=\A^{-1}\M$ and let us denote the dilated ends of $M$ by $M_{k,\A}=\A^{-1}M_{k}$. Next, for a smooth function $h$ defined in $\M$, we consider dilated and translated Fermi coordinates
$$
X_{\A,h}(y,t):=X(\tilde{y},z), \quad \tilde{y}=\A y, \quad z=\A (t+h(\A y))
$$
for $y\in M_{\A}$ and $|t+h(\A y)|<\frac{\eta}{\A}+\frac{\D}{\A}\log(2 + r(\A y))$.

\medskip
Scaling and translating expression \eqref{lap3} we obtain
\begin{multline}\label{lap4}
\A^2 \Delta_{X}=\Delta_{X_{\A,h}}=\pp_{tt} + \Delta_{M_\A} -\A^2\{\Delta_{M}h + |A_{M}|^2h\}\pp_{t} - \A^2|A_{M}|^2t\pp_{t}
\\
-2\A \,a_{ij}^0 \pp_{i}h\,\pp_{jt} + \A^2\,a_{ij}^0\pp_{i}h\pp_{j}\,h\pp_{tt} +  D_{\A,h}
\end{multline}
where
\begin{eqnarray}\label{Rest1}
D_{\A,h} & = &   \A (t + h) a_{ij}^1(\A y,\A (t+h) ) \, (\pp_{ij} - 2\A
\pp_i h\,\pp_{it} - \A^2 \pp_{ij} h\pp_t +  \A^2 \pp_{i}h\,\pp_{j}h\,
\pp_{tt})
\nonumber\\
& + & \A^2 (t + h) b_i^1(\A y,\A  (t+h) )\, (\pp_i - \A \pp_{i} h\,\pp_t)
\nonumber\\
& + & \A^4 (t + h)^3 b_3^1(\A y,\A (t+h))\,\pp_t.
\end{eqnarray}

Expression \eqref{lap4} holds true in the region
\begin{equation}\label{regionFermiCoord}
\NN_{\A,h}:=\A^{-1}\Omega\cap\left\{y + (t + h(\A y))\nu(\A y)\,:\, |t+ h|< \frac{\eta}{\A} + \frac{\D}{\A}\log(2 + r(\A y))\right\}
\end{equation}
and we will use it to handle equation \eqref{SPAC} well inside the region $\A^{-1}(\Omega \cap \NN)$.

\medskip
In this regard, observe that our geometrical considerations above do not take into account the effect of $\pp \Omega \cap \NN$. We need to consider a new system of coordinates to handle boundary related computations.

\medskip
In this part, we will use that the entire surface $M$ is orthogonal to $\pp \Omega$. Using assumption (iii) from the introduction, for fix $k=1,\ldots,m$ and in $M_{k}$ we can assume that the closed simple curve $C_k:=M_{k}\cap \Omega$ is parameterized by a mapping
$$
\theta \in (0,l_{k}) \mapsto \gamma_k:=\gamma_k(\vv).
$$

The mapping $\gamma_k$ in turn, allows a smooth orthogonal extension to an open neighborhood of $C_{k}$ in $M_k$. Abusing the notation we write this extension as
$$
(\rho,\vv)\in \to \gamma_k=\gamma_{k}(\rho,\vv), \quad \rho\in (-\D,\D),\quad \theta\in(0,l_k)
$$
which can be chosen satisfying
$$
\gamma_k(0,\vv)=\gamma_k(\vv),\quad \pp_{\vv}\gamma_k(0,\vv)=\pp_{\vv}\gamma_{k}(\vv), \quad \pp_{\rho}\gamma_k \perp \pp_{\vv}\gamma_k
$$
 and such that $\gamma\left([0,\D)\times(0,l_{k})\right)\subset M_{k}\cap\bar{\Omega}$. In essence, the coordinates $(\rho,\vv)$ work as polar coordinates in $M_k$ near $C_k$.

\medskip
In the coordinates $\gamma_k(\rho,\vv)$ and ommiting the explicit depedence on $k$, the laplace beltrami operator of $M$ close to $C_k$, takes the form
\begin{equation}\label{LAPBELTOPR3}
\Delta_{M} = a^0_{ij}\pp_{ij} + b^0_i\pp_j, \quad i,j=\rho, \theta
\end{equation}
and where
$$
a_{\rho\rho}^0(\rho,\theta)=|\pp_{\rho}\gamma_k|^{-2}, \quad a_{\vv\vv}^0(\rho,\vv)=|\pp_{\vv}\gamma_k|^{-2}, \quad a_{\rho\vv}^0=a_{\vv\rho}^0=0.
$$

$$
b_{\rho}^0(\rho,\vv)=|\pp_{\rho}\gamma_k|^{-2}|\pp_{\vv}\gamma_k|^{-2}<\pp_{\rho\vv}\gamma_k;\pp_{\vv}\gamma_k>+|\pp_{\rho}\gamma_k|^{-4}<\pp_{\rho\rho}\gamma_k;\pp_{\rho}\gamma_k>
$$

$$
b_{\vv}^0(\rho,\vv)=|\pp_{\rho}\gamma_k|^{-2}|\pp_{\vv}\gamma_k|^{-2}<\pp_{\rho\vv}\gamma_k;\pp_{\rho}\gamma_k>+|\pp_{\rho}\gamma_k|^{-4}<\pp_{\vv\vv}\gamma_k;\pp_{\vv}\gamma_k>.
$$

\medskip
Proceeding as above and associated to the coordinate system $y=\gamma_k(\rho,\vv)$ we consider Fermi coordinates
$$
X_k(\rho,\vv,z):=\gamma_k(\rho,\vv)+z\nu(\rho,\vv)
$$
in the neighborhood of $M_k$
$$
\NN:=\left\{\gamma_k(\rho,\vv) + z\nu(\rho,\vv)\,:\, |z|<\eta + \D\log\left(2 + r(\gamma_k(\rho,\vv))\right), \quad |\rho|<\D, \quad \vv\in (0,l_{k})\right\}.
$$

To described $\Omega \cap \NN$ near $C_k$ we assume the existence of a smooth funtion $G_{k}=G_{k}(\vv,z)$ such that $G_{k}(\vv,0)=0$

\medskip
Making a translation of the integral lines of $M_k$ associated to the parameterization $\gamma_k$ in the $\rho$ direction using the function $G_k$, namely
$$
\rho(s,\vv,z): = s + G_{k}(\vv,z), \quad |s|<\D, \quad \vv \in (0,l_k), \quad |z|<\eta
$$
we can described
$\Omega\cap\NN$ using the modified Fermi coordinates
$$
\tilde{X}(s,\vv,z):=\gamma_{k}(\rho(s,\vv,z),\vv)+ z\nu(\rho(s,\vv,z),\vv)
$$
which actually provide a change of variables in the larger set
$$
\tilde{\NN}:=\left\{x=\tilde{X}(s,\theta,z)\,:\, |z|<\eta + \D \log(2 +r(\gamma_k(\rho(s,\vv,z),\theta))), \quad |s|<\D, \quad \theta \in (0,l_{k})\right\}
$$
and clearly
$$
\pp \Omega \cap \NN = \{\gamma_{k}(G_k(\vv,z),\theta) +z\nu(G_k(\vv,z)\,:\, \theta \in (0,l_k), \quad |z|< \eta + \D \log\left(2 + r(\gamma_{k}\left(G_k(\vv,z),\vv
\right))\right)\}.
$$

\medskip
Observe that
$$
\pp_{z}\tilde{X}(0,\vv,0)= \nu(0,\vv)+  \pp_{\rho}\gamma_{k}(0,\vv)\cdot \pp_{z}G_{k}(\vv,0)
$$
and
$$
\pp_{\vv} \tilde{X}(0,\vv,0)= \pp_{\theta}\gamma_{k}(0,\vv), \quad \pp_{s}\tilde{X}(0,\vv,0)=\pp_{\rho}\gamma_{k}(0,\vv)
$$
so that, by making the inner product of $\pp_z \tilde{X}(0,\vv,0)$ against $\pp_{s} \tilde{X}(0,\vv,0)$, we observe that conditon (I) is rewritten as $\pp_{z}G_k(\vv,0)=0$.

\medskip
Summarizing, we assume on the function $G_k(\vv,z)$ that
\begin{equation}\label{CONDBDR}
G_{k}(\vv,0)=0, \quad \pp_{z}G_k(\vv,0)=0, \quad \vv \in (0,l_{k})
\end{equation}

\medskip
From \eqref{CONDBDR}, the asymptotic expansion in powers of $z$ of the mapping $\tilde{X}(s,\vv,z)$ reads as
\begin{equation}\label{MODFFERMICOORD}
\tilde{X}(s,\vv,z):= \gamma_{k}(s,\vv) + z \nu(s,\vv) + \frac{z^2}{2} q_1(s,\vv) + \frac{z^3}{6}q_2(s,\vv) + \OO(z^4)
\end{equation}
where $q_1,q_2\perp \nu$ with expressions given by
$$
q_1(s,\vv):= \pp_{\rho}\gamma_{k}(s,\vv)\cdot\pp_{zz}G_k(\vv,0),\quad
q_2(s,\vv)= \pp_{\rho}\gamma_{k}\cdot\pp_{z}^{(3)}G_k(\vv,0) + 3 \pp_{\rho}\nu\cdot\pp_{zz}G_{k}(\vv,0).
$$

Taking derivaties in expression \eqref{MODFFERMICOORD} and ommiting the dependence on $k$, we can compute the induced metric of $M$, which in this coordinates takes the form
\begin{equation}\label{boundarymetric}
\tilde{g}=\left(
\begin{array}{ccc}
|\pp_{\rho}\gamma|^2&0&0\\
0&|\pp_{\vv}\gamma|^2|&0\\
0&0&1
\end{array}
\right)+z\left(
\begin{array}{ccc}
-2\tilde{M}&-2\tilde{N}&<\pp_{\rho}\gamma;q_1>\\
-2\tilde{N}&-2\tilde{R}&<\pp_{\vv}\gamma;q_1>\\
<\pp_{\rho}\gamma;q_1>&<\pp_{\vv}\gamma;q_1>&0
\end{array}
\right)+\OO(z^2)
\end{equation}
where
$$
-\tilde{M}=<\pp_{\rho} \gamma;\pp_{\rho} \nu>, \quad -\tilde{R}=<\pp_{\vv} \gamma;\pp_{\vv} \nu>, \quad -2\tilde{N}=<\pp_{\rho} \gamma;\pp_{\vv} \nu>+<\pp_{\vv}\gamma;\pp_{\rho} \nu>
$$
where all the entries of the matrices above are evaluated at $(s,\vv)$.

\medskip
Consequently, the inverse of the metric has the asymptotic expression $\tilde{g}^{-1}=(\tilde{g}^{ij})_{3\times 3}$ has the form
\begin{multline}\label{boundarymetric-1}
\tilde{g}^{-1}=\left(
\begin{array}{ccc}
|\pp_{\rho}\gamma|^{-2}&0&0\\
0&|\pp_{\vv}\gamma|^{-2}|&0\\
0&0&1
\end{array}
\right)\\
+z\left(
\begin{array}{ccc}
2|\pp_{\rho} \gamma|^{-4}\tilde{M}&2|\pp_{\rho}\gamma|^{-2}|\pp_{\vv}\gamma|^{-2}\tilde{N}&-|\pp_{\rho} \gamma|^{-2}<\pp_{\rho}\gamma;q_1>\\
2|\pp_{\rho}\gamma|^{-2}|\pp_{\vv}\gamma|^{-2}\tilde{N}&2|\pp_{\vv} \gamma|^{-4}\tilde{R}&-|\pp_{\vv} \gamma|^{-2}<\pp_{\vv}\gamma;q_1>\\
-|\pp_{\rho} \gamma|^{-2}<\pp_{\rho}\gamma;q_1>&-|\pp_{\vv} \gamma|^{-2}<\pp_{\vv}\gamma;q_1>&0
\end{array}
\right)+\OO(z^2)
\end{multline}

\medskip
We consider again the function $h$ as above and proceeding in the same fashion we take dilated and translated Modified Fermi coordinates
$$
\tilde{X}_{\A,h}(s,\theta,t)= \A^{-1}\tilde{X}(\A s,\A \theta, \A (t+ h(\A s,\A \theta)))
$$
for
$$
0<s< \frac{\D}{\A} , \quad \theta \in (0, \frac{l_{k}}{\A}), \quad |z|< \frac{\eta}{\A} + \frac{\D}{\A}\log(2 + r(\gamma(\A s, \A \theta))).
$$

\medskip
After a series of lenghty but necessary computations we arrive to the expression for the euclidean laplacian in the coordinates $\tilde{X}_{\A,h}$
\begin{equation}\label{lapbdr}
\A^2\Delta_{\tilde{X}}=\Delta_{
\tilde{X}_{\A,h}}= \pp_{tt} + \Delta_{M_{\A}} -\A^2|A_{M}|^2 \,t\pp_{t} - \A^2\left\{\Delta_{M} h +|A_{M}^2h\right\}\pp_{t} +D_{0} + \tilde{D}_{\A,h}
\end{equation}
where following \cite{GuoJunYang}, we denote
$$
l_1(\vv)=|\pp_{\vv}\gamma(0,\vv)|>0, \quad l_2(\vv)=|\pp_{\rho}\gamma(0,\vv)|>0,\quad  I(\vv)= l_2(\vv)^2\,\pp_{zz}G_{k}(\vv,0)
$$

$$
 A(\vv)=<\pp_{\rho\rho}\gamma(0,\vv);\pp_{\rho}\gamma(0,\vv)>
,\quad C(\vv)=<\pp_{\rho\vv}\gamma(0,\vv);\pp_{\vv }\gamma(0,\vv)>
$$

$$
R(\vv)= <\pp_{\vv\vv}\gamma(0,\vv);\pp_{\vv}\gamma(0,\vv)>
,\quad E(\vv)=<\pp_{\rho\vv}\gamma(0,\vv);\pp_{\rho}\gamma(0,\vv)>, \quad
$$
to find that
\begin{eqnarray}\label{BDRREST0}
\tilde{D}_0&=&-2\A \frac{I(\A \theta)}{l_1(\A\theta)l_2(\A \theta)}(t+h)\left(\pp_{st}
- \A \pp_{\rho} h \pp_{tt}\right)- \A l^{-2}_1(\A \theta)\pp_{\vv\vv}h \pp_{\theta t}  - \A l^{-2}_2(\A \theta)\pp_{\rho\rho} h\pp_{st}\nonumber\\
&&+\A^2l^{-2}_1(\A \theta)|\pp_{\vv} h|^2\pp_{tt} + \A^2 l^{-2}_2(\A \theta)|\pp_{\rho} h|^2\pp_{tt}- 2\A^3 sl_2^{-4}(\A \theta)A(\A \theta)|\pp_{\rho}h|^2\pp_{tt}
\end{eqnarray}
and
\begin{eqnarray}
\tilde{D}_{\A,h}&=& \A^2\tilde{a}_1\left(\pp_{\theta t} - \A \pp_{\vv}h \pp_{tt}\right)+\A^2\tilde{a}_2
\left(\pp_{s t} - \A \pp_{\rho}h \pp_{tt}\right)\nonumber\\
&+& \A^2\tilde{b}_1\left(\pp_{\theta } - \A \pp_{\vv}h \pp_{t}\right)+\A^2\tilde{b}_2
\left(\pp_{s} - \A \pp_{\rho}h \pp_{t}\right) \nonumber\\
&+& \A^3(t+h)^2\tilde{b}_3(\A s,\A\theta,\A (t+h))\pp_t + \A^4\tilde{R}_{\A,}\label{BDRREST1}
\end{eqnarray}
where the functions $\tilde{a}_1,\tilde{a}_1,\tilde{b}_1,\tilde{b}_2, \tilde{b}_3$ are smooth with bounded derivatives and $\tilde{R}_{\A}$ is a differential operator having $C^1$ dependence on $h$ and its derivatives.

\medskip
We remark that in the coordinates $(s,\theta)$

$$
\Delta_{M_{\A}} =
\frac{1}{l^2_1(\A \theta)}\pp_{\theta\theta} + \frac{1}{l^2_2(\A \theta)}\pp_{\rho\rho} - 2\A s \frac{A(\A \theta)}{l^4_2(\A \theta)}\pp_{ss}
+
\A\left(\frac{C(\A \theta)}{l_1^2(\A \theta)l_2^2(\A \theta)}- \frac{A(\A \theta)}{l^4_2(\A \theta)}\right)\pp_{s}
$$

\begin{equation}\label{DILATLAPBELBDR}
+\A\left(\frac{E(\A \theta)}{l_1^2(\A \theta)l_2^2(\A \theta)}- \frac{R(\A \theta)}{l^4_1(\A \theta)}\right)\pp_{\theta} + \OO(\A^2)
\end{equation}
\medskip

Next, we turn our attention to the boundary condition. It can be check directly from \eqref{boundarymetric} and \eqref{boundarymetric-1} that the boundary condition reads as
$$
\frac{\pp}{\pp n} = \sqrt{\tilde{g}^{11}}\pp_{\rho} + \frac{\tilde{g}^{12}}{\sqrt{\tilde{g}^{11}}}\pp_{\vv}+
\frac{\tilde{g}^{13}}{\sqrt{\tilde{g}^{11}}}\pp_{z}
$$
so that, after dilating and translating we find that the boundary condition becomes
$$
\pp_{s}=\A I(\A \theta) t\pp_{t} +\A\left\{\pp_{\rho} h + I(\A \theta)h\right\}\pp_{t}
$$
$$
- \left[2\A (t+h)m_1(\A \theta) + \A^2(t+h)^2\tilde{d}_1(\A \theta)\right]\pp_{s}
+\left[\A(t+h)m_2(\A \theta) + \A^2 \tilde{d}_2(\A \theta,\A(t+ h))\right]\pp_{\theta}
$$
\begin{equation}\label{bdrcondition}
\left[ \A^2(t+h)^2 I(\A \theta)m_1(\A \theta)+2\A^2 m_1(\A \theta)\pp_{\rho}h(t+h) - \A^2 m_2(\theta)\pp_{\vv}h(t+h) + \A^3\tilde{B}_{\A}
\right] \pp_{t}
\end{equation}

where
$$
m_1(\vv) = |l_2(\vv)|^{-2}<\pp_{\rho} \nu(0,\vv) ; \pp_{\rho} \gamma(0,\vv)>
$$

$$
m_2(\vv)= |l_1(\vv)|^{-2}\left(<\pp_{\vv} \nu(0,\vv) ; \pp_{\rho} \gamma(0,\vv)> + <\pp_{\rho} \nu(0,\vv) ; \pp_{\vv} \gamma(0,\vv)>\right)
$$
and
$\tilde{B}_{\A}= \tilde{B}_{\A}(\theta,t,h,\nabla_{\M} h)$
is has $C^1$ dependence in its variables.

\medskip
We mention that in the case of the catenoid and an axially symmetric  domain, following the scheme in \cite{Kowalczyk, delPinoKowalczykWei} one can parameterize with only one set of coordinates and the calculations reduce considerably.

\section{the approximation and preliminary discussion}\label{APPROXSOL}

The proof of our main result relies on a Lyapunov-Schmitt procedure near an almost solution to the equation. This section is devoted to find a good global approximation to perform this reduction.

\medskip
For this, we denote $f(u)=u(1-u^2)$ and we consider the solution of the ode
$$
w''(s) +w(1-w^2)=0,\quad  s\in \R, \quad w(\pm \infty)=\pm 1
$$
which is given explicitly by
$$
w(s) \,=\, \tanh\left(\frac{s}{\sqrt{2}}\right),\quad s\in\R
$$
and has the asymptotic properties
\begin{equation}\label{Heteroclinic Asymptotics}
\begin{array}{cccc}
w(s) &=& 1 \,-\, 2 \, e^{-\sqrt{2}\,s} \,+\,
\OO\left(e^{-2\sqrt{2}|s|}\right),&s>1\\
w(s) &=& -1 \,+\, 2\,e^{\sqrt{2}\,s} \,+\,
\OO\left(e^{-2\sqrt{2}|s|}\right),& s<-1\\
w'(s)&=& 2\sqrt{2}\,e^{-\sqrt{2}\,|s|} \,+\,
\OO\left(e^{-2\sqrt{2}|s|}\right),& |s|>1
\end{array}
\end{equation}
where $w'=\frac{dw}{ds}$.

Next, denoting $\Omega_{\A}:=\A^{-1} \Omega$, and after an obvious rescaling we are lead to consider the problem
\begin{equation}\label{rescaledACEqn}
\Delta u + f(u)=0,\quad \hbox{in }\Omega_{\A}, \quad \frac{\pp u}{\pp n_{\A}}=0, \quad \hbox{on } \pp \Omega_{\A}
\end{equation}
where $n_{\A}$ stands for the inward unit normal vector to $\pp \Omega_{\A}$.

\medskip
In what follows for a function $U=U(x)$ and a subdomain $\Omega'\subset\Omega_{\A}$, we denote  by
$$
S(U)=\Delta U +f(U), \quad \hbox{in }\Omega'.
$$

\subsection{The inner approximation}\label{innerapprox}
Let us take $h\in W^{2,p}(\M)$ satisfying the apriori estimate.
\begin{equation}\label{APRIORIESTNODALSET}
\|D^2 h \|_{L^p(\M)} + \|\nabla h\|_{L^{\infty}(\M)} + \|h \|_{L^{\infty}(\M)}\leq \mathcal{K} \A
\end{equation}
where the constant $\mathcal{K}$ is going to be chosen large but independent of $\A>0$.

\medskip
Using the coordinates $X_{\A,h}$ we set as first local approximation
$$
u_0(x)=w(t), \quad x=X_{\A,h}(y,\theta,t)\in \A^{-1}(\Omega\cap \NN).
$$

When computing the error created by $u_0$ using \eqref{lap4}-\eqref{Rest1}, in $\A^{-1}(\Omega\cap\NN)$ we find that
$$
S(u_0)= -\A^2\{\Delta_{M}h + |A_{M}|^2h\}w'(t) - \A^2|A_{M}|^2tw'(t) + \A^2\pp_{i}h\pp_{j}h w''(t)
$$

$$
-\A^3(t+h)a_{ij}^1(\A y,\A (t+h))\left(\pp_{ij}h w'(t) - \pp_{i}h\pp_{j}h w''(t)\right) - \A^3b_{i}^1(\A y,\A (t+h))\pp_{i}h w'(t)
$$

\begin{equation}\label{ErrorU1}
- \A^4(t+h)b_{3}^1(\A y,\A(t+h))w'(t)
\end{equation}
where $|A_{M}|, h,\pp_{i}h,\pp_{ij}h$ are evaluated at $\A y$.

\medskip
Observe that if we take $h=0$, the size and behavior of the error in expression \eqref{ErrorU1} is given by
$$
-\A^2|A_{M_0}|^2tw'(t) + \A^4b_{3}^1(\A y,\A t)t^3w'(t).
$$

\medskip
As in \cite{9}, due to the presence of the $\OO(\A^2)$ term we need to improve this approximation. Hence we consider the function $\psi_1(t)$ solving the ode
\begin{equation}\label{g1Integral}
\pp_{tt}\psi_1(t) + F'(w(t))\psi_1(t)=tw'(t), \quad t\in \R.
\end{equation}

\medskip
Using variations of parameters formula and the fact that
$$
\int_{\R}t (w'(t))^2dt=0
$$
we obtain that $\psi_1(t)$ given by the formula
$$
\psi_1(t)=-w(t)\int_{0}^t w'(s)^{-2} \int_{s}^{\infty}\xi
w'(\xi)^2\,d\xi ds
$$
from where it follows at once that
$$
\|e^{\sigma|t|}\pp_{t}^{(j)}\psi_1\|_{L^{\infty}(\R)}\leq C_j,
\quad j\in \mathbb{N}, \quad 0<\sigma<\sqrt{2}.
$$

So, we consider as a second approximation in the region $\A^{-1}(\Omega\cap\NN)$ the function
\begin{equation}\label{SndApprox}
u_1(x) \,=\, w(t) \,+\, \phi_{1}(y,t)
\end{equation}
where in the coordinates $X_{\A,h}$
\begin{equation*}
\phi_{1}(y,t)\,=\,\A^2|A_{M}(\A
y)|^2\,\psi_1(t).
\end{equation*}

When computing the inner error of this new approximation we find that
$$
S(u_1)=\Delta\phi_1 + f(w(t))\phi_1 + S(u_0) + f(w(t) + \phi_1) -f(w(t)) - f'(w(t))\phi_1
$$

$$
= -\A^2\{\Delta_{M}h + |A_{M}|^2h\}w'(t) + \A^2\pp_{i}h\pp_{j}h w''(t)
-\A^3(t+h)a_{ij}^1(\A y,\A (t+h))\left(\pp_{ij}h w'(t) - \pp_{i}h\pp_{j}h w''(t)\right)
$$

$$
- \A^3b_{i}^1(\A y,\A (t+h))\pp_{i}h w'(t) - \A^4(t+h)b_{3}^1(\A y,\A(t+h))w'(t)
$$

$$
+\A^4\Delta_{M}(|A_{M}|^2)\psi_1(t) - \A^4\{\Delta_{M}h + |A_{M}|^2h\}|A_{M}|^2\psi(t) + \A^4|A_{M}|^4t\pp_{t}\psi_1(t)
$$

$$
- 2\A^4a_{ij}^0(\A y)\pp_{i}h\pp_{j}(|A_{M}|^2)\pp_{t}\psi_1(t) +\A^4a_{ij}^0(\A y)\pp_{i}h\pp_{j}h|A_{M}|^2\pp_{tt}\psi_1(t)
$$

\begin{equation}\label{ErrorU11}
N(\phi_1)+ \A^5R_{1,\A}(\A y,t,h,\nabla_{M}h,D^2_{M}h)
\end{equation}
where
$$
N(\phi_1)=f(w(t))\phi_1 + S(u_0) + f(w(t) + \phi_1) -f(w(t)) - f'(w(t))\phi_1\sim \OO(\A^4 e^{-\sigma|t|})
$$
and the differential operator has $C^1$ dependence of all of its variables with
$$
|\nabla R_{1,\A}|+ |R_{1,\A}|\leq C e^{-\sigma|t|}, \quad \hbox{for }0<\sigma<\sqrt{2}.
$$

\medskip
From the error \eqref{ErrorU11} we see that in $\A^{-1}(\Omega\cap\NN)$ the open neighborhood of $\M$
\begin{equation}\label{innerErrorBehavior}
|S(u_1) -\A^2\{\Delta_{M}h + |A_{M}|^2h\}w'(t)| \leq C\A^4e^{-\sigma |t|}
\end{equation}

\medskip
\subsection{Boundary correction}\label{boundarycorrect}

It is clear that our approximation $u_1$ can be defined in the set $\A^{-1}(\Omega\cap \NN)$, but $u_1$ does not satisfy in general the boundary condition. In this regard we need to make a further improvement of the approximation $u_1$ by adding boundary correction terms.

\medskip
Let us consider a cut-off function $\beta$ such that
$$
\beta(\rho) =\left\{
\begin{array}{ccc}
1,& 0\leq \rho<\frac{\D}{2}\\
0,& \rho>\D
\end{array}
\right.
$$

\medskip
For the $k-th$ end of $\A^{-1}M$, $M_{k,\A}$ we consider a cut-off function $\beta_{\A}=\beta_{k,\A}(x)= \beta(\A s)$ for $x=\tilde{X}_{\A,h}(s,\theta,t)$.

\medskip
Near the boundary we consider a new approximation of the form
$$
u_2(x)=u_1(x) + \sum_{k=1}^m\left(\beta_{\A,k}\phi_{2,k}(x) + \beta_{\A,k}\phi_{3,k}(x)\right)
$$
where $\phi_{2,k}$, $\phi_{3,k}(x)$ are going to be chosen of order $\OO(\A)$ and $\OO(\A^2)$ respectively. To see how to choose $\phi_{2,k},\phi_{3,k}$, we proceed first by computing the error of the boundary condition created by the approximation $u_2$ using the coordinates $\tilde{X}_{\A,h}$ and expression \eqref{bdrcondition} for a fixed end $M_k$. In what follows we ommit the explicit dependence of $k$ but we remark that the developments in this part hold true regardless of the end we are working with, since they all have the same geometrical structure and the supports of the cutt-off functions $\beta_{\A,k}$ within te region $\NN_{\A,h}$ close to every end $M_{k,\A}$ are far away from each other.

\medskip
Recalling that we have assumed $h=\OO_{W^{2,p}(\M)}(\A)$ and splitting the boundary error in powers of $\A$ we find
$$
\tilde{B}(u_2)= -\pp_{s}\phi_2 + \A I(\A \theta)tw'(t) - \pp_{s}\phi_{3}
$$

$$
+ \A\left(\pp_{\rho}h_1 +I(\A \theta)h_1\right)w'(t) - \A^2 I(\A\theta)m_1(\A \theta)t^2w'(t) +\A I(\A \theta)t \pp_{t}\phi_2 - 2\A m_1(\A \theta)t\pp_{s}\phi_2
$$

$$
- \A^3\pp_{\rho}(|A_{M}|^2)\psi_1(t) + \A^3 I(\A \theta)|A_{M}|^2\pp_{t}\psi_1(t)
$$

$$
+\A\left\{\pp_{\rho}h + I(\A \theta)h\right\}\pp_{t}\phi_2 -2m_1(\A \theta)h\pp_{s}\phi_2 - \A^2(t+h)^2\tilde{d}_1(\A \theta)\pp_{s}\phi_2
$$

\begin{equation}\label{bdrerror0}
-2\A^2 I(\A \theta)m_1(\A \theta)hw'(t) - \A^2I(\A \theta)m_1(\A \theta)t^2\pp_{t}\phi_2 +  \A^4\tilde{B}_{0,\A}
\end{equation}
where the term $\tilde{B}_{0,\A}$ satisfies that
$$
|\nabla \tilde{B}_{0,\A}|+|\tilde{B}_{0,\A}|\leq Ce^{-\sigma|t|}
$$
Our goal is to get a boundary error of order $\OO(\A^3e^{-\sigma}|t|)$.  In order to do so, we first choose $\phi_{2}$ solving the equation
$$
\pp_{tt} \phi_{2} + \Delta_{M_{\A}} \phi_{2} +f'(w(t))\phi_{2} =0, \quad \hbox{in }M_{\A}\times \R
$$
$$
\pp_{s} \phi_{2} =\A I(\A \theta)tw'(t)
$$
from where we obtain that $\phi_{2}(\cdot,t)$ is odd in $t$ and from proposition \ref{ProjectedProblemInvertTheory} it follows that in the norms \eqref{error norm}-\eqref{norm}
$$
\|D^2 \phi_2\|_{p,\sigma} + \|e^{\sigma|t|}\nabla \phi_{2}\|_{L^{\infty}(M_{\A}\times\R)} + \|e^{\sigma|t|}\phi_{2}\|_{L^{\infty}(M_{\A}\times\R)} \leq C\A.
$$

\medskip
To choose $\phi_{3}$, we make the decomposition
$$
t^2w'(t)= c_1\,w'(t) + g_1(t), \quad c_1=\|w'\|^{-2}_{L^2(\R)}\int_{\R}t^2(w'(t))^2dt
$$
$$
t\pp_{s}\phi_{21}(\cdot,t)= c_2(\cdot)w'(t) + g_2(\cdot,t), \quad \int_{\R}g_2(\cdot,t)w'(t)dt=0
$$
and we write the second line in \eqref{bdrerror0} as
$$
+\A\left(\pp_{\rho}h +I(\A \theta)h\right)w'(t) - \A^2 c_1 I(\A\theta)m_1(\A \theta)w'(t)
- 2\A m_1(\A \theta)c_2(\theta)w'(t)
$$

$$
- \A^2 I(\A\theta)m_1(\A \theta)g_1(t) - 2\A m_1(\A \theta)g_2(\theta,t) +\A I(\A \theta)t \pp_{t}\phi_2.
$$

\medskip
Hence to choose $\phi_3$ we first ask $\phi_3$ to satisfy the boundary condition on $\pp M_{\A}\times \R$
\begin{equation}\label{bdrcondPHI3}
\pp_{s}\phi_3(0,\theta,t)=- \A^2 I(\A\theta)m_1(\A \theta)g_1(t) - 2\A m_1(\A \theta)g_2(\theta,t) +\A I(\A \theta)t \pp_{t}\phi_2.
\end{equation}

Next we compute the error of the approximation $u_2$ near the boundary using expression \eqref{lapbdr}. Setting
$$
S_{out}(u_2) = S(u_1) + \Delta \phi_2 +f'(w(t))\phi_2 + \Delta \phi_3 + f'(w(t))\phi_3 + \frac{1}{2}f''(w(t))(\phi_2+ \phi_3)^2
$$

$$
[f'(u_1) -f'(w(t))](\phi_2 + \phi_3) + \frac{1}{2}[f''(u_1) -f''(w(t))](\phi_2 + \phi_3)^2
$$

$$
[f(u_1 + \phi_2 + \phi_3) -f(u_1) - f'(u_1)(\phi_2 + \phi_3) - \frac{1}{2}f''(u_1)(\phi_2 + \phi_3)^2].
$$

$$
\
$$
So that we can write this error explicitely in coordinates $(s,\theta,t)$ as

$$
S_{out}(u_2) = S(u_1)+ \pp_{tt} \phi_3 + \Delta_{M_{\A}}\phi_3 + f'(w(t))\phi_3 - \A^2|A_{M}|^2t\pp_{t}\phi_2 + \frac{1}{2}f''(w(t))\phi_2^2 - 2\A \frac{I(\A \theta)}{l_1(\A \theta)l_2(\A \theta)}t\pp_{st}\phi_{2}
$$

$$
-\A^2\left\{\Delta_{M}h + |A_{M}|^2h\right\}\pp_{t}\phi_{2} - 2\A \frac{I(\A \theta)}{l_1(\A \theta)l_2(\A \theta)}h\pp_{st}\phi_{2} +2\A^2\frac{I(\A \theta)}{l_1(\A \theta)l_2(\A \theta)}(t+h)\pp_{\rho}h\pp_{tt}\phi_{2}
$$

$$
-\A l_1^{-2}(\A \theta)\pp_{\vv\vv}h^2\pp_{\theta t}\phi_2
-\A l_2^{-2}(\A \theta)\pp_{\rho\rho}h\pp_{s t}\phi_2 + \A^2 l_1^{-2}(\A \theta)|\pp_{\vv}h|^2\pp_{tt}\phi_2
$$

$$
+ \A^2 l_2^{-2}(\A \theta)|\pp_{\rho\rho}h|^2\pp_{t t}\phi_2-2\A^3 sl_2^{-4}(\A \theta) A(\A \theta)|\pp_{\rho}h|^2
\pp_{tt}\phi_2
$$

$$
+ \A^2\tilde{a}_1(\A s,\A \theta, \A(t+h))\,\left\{\pp_{\theta t}\phi_2 - \A \pp_{\vv}h \phi_{tt}\phi_2\right\}  + \A^2\tilde{a}_2(\A s,\A \theta, \A(t+h))\,\left\{\pp_{s t}\phi_2 - \A \pp_{\rho}h \phi_{tt}\phi_2\right\}
$$

$$
+ \A^2\tilde{b}_1(\A s,\A \theta, \A(t+h))\,\left\{\pp_{\theta }\phi_2 - \A \pp_{\vv}h \phi_{t}\phi_2\right\}  + \A^2\tilde{b}_2(\A s,\A \theta, \A(t+h))\,\left\{\pp_{s}\phi_2 - \A \pp_{\rho}h \phi_{t}\phi_2\right\}
$$

\begin{equation}\label{BDRERROREXPLICIT}
+ f'(w(t))(\phi_2\cdot\phi_3 + \frac{1}{2}\phi_3^2)+ [f'(u_1)-f'(w(t))](\phi_2 + \phi_3)+ +\A^3(t+h)\tilde{b}_3(\A s,\A \theta,\A(t+h))\pp_{t}\phi_2 + \tilde{R}_{2,\A}
\end{equation}
where $\tilde{R}_{2,\A}=\tilde{R}_{2,\A}(\A s, \A\theta,t,h,\nabla_{\M}h,D^2_{\M}h)= \OO(\A^4)$ and
$$
|D\tilde{R}_{2,\A}|+|\tilde{R}_{2,\A}|\leq C\A^3.
$$
we notice that
$$
S(u_2)=(1-\beta_{\A})S(u_1) + \beta_{\A}S_{out}(u_2) + 2\nabla\beta_{\A}\cdot\nabla \phi_2 + 2 \nabla \beta_{\A}\cdot\nabla \phi_3 + (\phi_2 + \phi_3)\Delta \beta_{\A}
$$

$$
f(u_{1}+ \beta_{\A}(\phi_2 + \phi_3)) - f(u_1) - \beta_{\A}f(u_1 + \phi_2 + \phi_3) - (1-\beta_{\A})f(u_1)
$$
and observe that
\begin{equation}\label{bdrerrortermgluing}
\nabla \beta_{\A}\cdot \nabla \phi_2 = \A \pp_{s}\beta(\A s)\pp_{s}\phi_2 = \OO(\A^2e^{-\sigma|t|})
\end{equation}
and
$$
f(u_{1}+ \beta_{\A}(\phi_2 + \phi_3)) - f(u_1) - \beta_{\A}f(u_1 + \phi_2 + \phi_3) - (1-\beta_{\A})f(u_1) = \OO(\A^4e^{-\sigma|t|})
$$

Hence in order to improve the approximation we need to get rid of the terms in the first line of \eqref{BDRERROREXPLICIT} and the term in \eqref{bdrerrortermgluing} so that we choose $\phi_3$ solving the linear problem

$$
\pp_{tt} \phi_3 + \Delta_{M_{\A}}\phi_3 + f'(w(t))\phi_3 = \A^2|A_{M}|^2t\pp_{t}\phi_2 - \frac{1}{2}f''(w(t))\phi_2^2 + 2\A \frac{I(\A \theta)}{l_1(\A \theta)l_2(\A \theta)}t\pp_{st}\phi_{2}- \A\pp_{s}\beta(\A s)\pp_{s}\phi_2, \quad \hbox{in }M_{\A} \times \R
$$
\begin{equation*}
\pp_{s}\phi_3(0,\theta,t)=- \A^2 I(\A\theta)m_1(\A \theta)g_1(t) - 2\A m_1(\A \theta)g_2(\theta,t) +\A I(\A \theta)t \pp_{t}\phi_2.
\end{equation*}
$$
\int_{\R}\phi_3(\cdot,t)w'(t)dt=0, \quad \hbox{in }M_{\A}.
$$

From this, $\phi_3$ satisfies
$$
\|D^2 \phi_3\|_{p,\sigma} + \|e^{\sigma|t|}\nabla \phi_{3}\|_{L^{\infty}(M_{\A}\times\R)} + \|e^{\sigma|t|}\phi_{3}\|_{L^{\infty}(M_{\A}\times\R)} \leq C\A^2.
$$

From expression \eqref{BDRERROREXPLICIT} we directly check that
\begin{equation}\label{behaviorBDRERRORLAP}
S(u_2)=S(u_1) + E_{0,\A} + E_{1,\A}
\end{equation}
where $E_{i,\A}=E_{i,\A}(\A s, \A\theta,t,h,\nabla_{\M}h,D^2_{\M}h)= \OO(\A^{2+i})$ and
$$
|D E_{i,\A}|+|E_{i,\A}|\leq C\A^{3+i}, \quad i=0,1.
$$
and from \eqref{bdrerror0} and \eqref{bdrcondPHI3} the boundary error takes the form
\begin{equation}\label{behaviorBDRERRORNUEMANN}
\tilde{B}(u_2)=\A\left(\pp_{\rho}h_1 +I(\A \theta)h_1\right)w'(t) - \A^2 c_1 I(\A\theta)m_1(\A \theta)w'(t)
- 2\A m_1(\A \theta)c_2(\theta)w'(t)
\end{equation}

$$
\tilde{B}_{1,\A}+ \tilde{B}_{2,\A}
$$
where $\|c_2\|_{L^{\infty}(\pp M_{\A})}\leq C \A$ and
$\tilde{B}_{i,\A}=E_{i,\A}(\A s, \A\theta,t,h,\nabla_{\M}h,D^2_{\M}h)= \OO(\A^2+i)$ and
$$
|D \tilde{B}_{i,\A}|+|\tilde{B}_{i,\A}|\leq C\A^{2+i}, \quad i=1,2.
$$

\medskip
To get the right size of the boundary error we impose on $h$ the boundary condition
$$
\pp_{\rho}h +I(\vv)h = \A c_1 I(\vv)m_1(\vv)
+2m_1(\vv)c_2\left(\frac{\vv}{\A}\right)
$$
where we remark that $\|c_2\|_{L^{\infty}(M_{\A})}\leq C\A$ so that
$$
\|\A c_1 Im_1
+2m_1c_2\left(\frac{\cdot}{\A}\right)\|_{L^{\infty}(\pp \M)}\leq C\A.
$$

\medskip
\subsection{The global approximation}\label{globalapprox} Observe that so far, our
approximation is defined only in the open set $\A^{-1}(\Omega\cap \NN)$. In order to define a global approximation in $\Omega_{\A}$ let us define the function
\begin{equation}\label{Interpolation}
\mathbb{H}(x):=\left\{
\begin{array}{ccc}
+1,& x \in \A^{-1}S^{+}\\
-1,&x \in \A^{-1}S^{-}.
\end{array}
\right.
\end{equation}
which is clearly an exact solution of the equation wherever it is smooth.

\medskip
The idea to get a global approximation is to consider the approximation $u_2$ well inside $\NN_{\A,h}$, while outside $\NN_{\A,h}$ we interpolate
with the function $\mathbb{H}(x)$. In order to make this precise, let us take a non-negative function $\beta$ in $C^{\infty}(\R)$ such that
$$
\beta(s)=\left \{
\begin{array}{ccc}
1,& |s| \leq 1
0,&  |s| \geq 2
\end{array}
\right.
$$
and consider the following cut-off function in $\NN_{\A,h}$ given by
$$
\beta_{\eta}(x)= \beta (|t + h(\A y)| - \frac{\eta}{\A} -
\D \ln(r( \A y))+ 2), \quad x= X_{\A,h}(y,\theta, t) \in \NN_{\A,h}.
$$

\medskip
With the aid of this, we set up as approximation in $\Omega_{\A}$ the function
\begin{equation}\label{global-approximation}
U(x) = \beta_{\eta}(x)u_2(x) + (1 - \beta_{\eta}(x)\mathbb{H}(x), \quad x \in \Omega_{\A}.
\end{equation}
and we compute the new error created by this approximation as follows
$$
S(U) \,=\, \Delta U \,+\, f(U)\,=\,\beta_2S(u_2) + E
$$
where
$$
E = f(\beta_{\eta}U) - \beta_{\eta}f(U) +  2 \nabla
\beta_{\eta}\cdot \nabla u + u \Delta \beta_{\eta}.
$$
Using that $z =|t + h(\A y)|$ we see that the derivatives of
$\beta_1$ do not depend on the derivatives of $h$. On the other hand, due to the choice of $\beta_{\eta}$ and the explicit form of
$E$, the error created only takes into account the values of
$\beta_{\eta}$ in the set
$$
x=X_{\A,h}(y,\theta,t) \in \NN_{\A,h}, \quad |t + h(\A y)| \geq
\frac{\eta}{\A} + 4\ln(r(\A y)) - 2,
$$
so we get the following estimate for the error $E$
$$
|E|\leq Ce^{-\frac{\eta}{\A}}.
$$

\section{ The proof of Theorem \ref{theo1}}\label{PROOFTHEO1}
The proof of Theorem \ref{theo1} is fairly technical. So,  to keep the presentation as clear as possible, we sketch the steps of the proof and in the next sections we give the detailed proofs of
the lemmas and propositions mentioned here.

\medskip
First, we introduce the norms we consider to set up an appropriate functional analytic scheme for the proof of Theorem \ref{theo1}.
For $\A >0$, $1 < p \leq \infty$ and a function $f(x)$,
defined in $\Omega_{\A}$, we set

\begin{equation}\label{gluing norm}
\|f\|_{p,\ts}:= \sup_{x \in \R^3}\|f\|_{L^p(B_1(x))}.
\end{equation}

We also consider for functions $g=g(y,t)$,
$\phi=\phi(y,t)$, defined in the whole $M_{\A}\times \R$, the norms
\begin{equation}\label{error norm}
\|g\|_{p,\sigma}:=\sup_{(y,t)\in M_{\A}\times \R}e^{\sigma|t|}\|g\|_{L^p(B_1(y,t);dA_{\A})}
\end{equation}

\begin{equation}\label{norm}
\|\phi\|_{2,p,\sigma} := \|D^2\phi\|_{p,\sigma} +
\|D\phi\|_{\infty,\sigma}+\|\phi\|_{\infty,\sigma}.
\end{equation}

where $dA_{\A} := dy_{g_{M_{\A}}}dt$. Of course, in the case $p= +\infty$, we have that $L^{\infty}(B_1(y,t)),dA_{\A})= L^{\infty}(B_1(y,t))$.

\medskip
While for a function $G$ defined in $\pp M_{\A}\times \R$ we consider the norm
$$
\|G\|_{p,\sigma}:= \|e^{\sigma |t|}G\|_{L^{p}(\pp M_{\A}\times \R)}
$$
\medskip
And we mentioned in the previous section, we set the norm for the parameter $h$ as
\begin{equation}\label{normnodalset}
\|h\|_{*}= \|D^2 h\|_{L^{p}(\M)} + \|\nabla h\|_{L^{\infty}(\M)} +
\|h\|_{L^{\infty}(\M)}
\end{equation}

\medskip
We look for a solution to equation
\eqref{SPAC} of the form
$$
u_{\A}(x) = U(x) +\varphi(x)
$$
where $U(x)$ is the global approximation defined in
\eqref{global-approximation} and $\varphi$ is going to be chosen
small in some appropriate sense. Thus, we need to solve the problem
$$
\Delta \varphi + f'(U)\varphi + S(U) + N(\varphi) = 0
$$
or equivalently
\begin{equation}
\begin{array}{lll}
\Delta \varphi + f(U)\varphi &=& - S(W) - N(\varphi)\\
\\
&=& -\beta_{\eta}S(u_2) - E - N(\varphi)
\end{array}
\label{LinearEquationforW}
\end{equation}
where
$$
N(\varphi) = f(U + \varphi) - f(U) - f'(U)\varphi.
$$

\medskip
\subsection{The Gluing Procedure.}\label{gluingproc} In order to solve problem
\eqref{LinearEquationforW}, we consider again the cut-off function
$\beta$, from previous section, and we define for every $n\in
\mathbb{N}$, the cut-off function
\begin{equation}\label{cut-off zeta}
\zeta_{n}(x) := \left\{
\begin{array}{ccc}
 \beta(|t + h(\A y)| - \frac{\eta}{\A} + n) &\quad
\hbox{ if } \quad x=X_{\A,h}(y_1,y_2,t)\in \NN_{\A,h} &\\
0 &\quad \hbox{ if } \quad x\in \NN_{\A,h}.&
\end{array}
\right.
\end{equation}

\medskip
We look for a solution to \eqref{LinearEquationforW} $\varphi(x)$ with the particular form
$$
\varphi(x) = \zeta_2(x)\phi(y,t) + \psi(x)
$$
where $\phi(y,t)$ is defined for every $(y,t) \in M_{\A}\times \R$ and
$\psi(x)$ is defined in the whole $\Omega_{\A}$. So, we find from equation
\eqref{LinearEquationforW} that
$$
\zeta_2\left[ \Delta_{\NN_{\A,h}} \phi +f'(U)\phi +
\zeta_1 U\psi + S(U) + \zeta_2N(\phi + \psi) \right]
$$
$$
+\, \Delta \psi - [2- (1 - \zeta_2)[f'(U)+2]]\psi + (1-\zeta_2)
S(U)
$$
$$
+\, 2 \nabla \zeta_2\cdot \nabla_{\NN_{\A,h}}\phi + \phi \Delta
\zeta_2 + (1- \zeta_2)N[\zeta_2\phi + \psi] =0.
$$

\medskip
Hence, we will have constructed a solution to the problem
\eqref{LinearEquationforW}, if solve the system
\begin{equation}
\Delta_{\NN_{\A,h}} \phi +f'(U)\phi + \zeta_2 U\psi +
S(U) + \zeta_2N(\phi + \psi) =0, \quad \hbox{in } |t+h(\A y)| <
\frac{\eta}{\A} -1 \label{System1}
\end{equation}

$$
\Delta \psi - [2- (1 - \zeta_2)[f'(U)+2]]\psi + (1-\zeta_2) S(U)
$$
\begin{equation}
+\, 2 \nabla \zeta_2\cdot \nabla_{\NN_{\A,h}}\phi \,+\, \phi \Delta
\zeta_2\,+\, (1- \zeta_2)N[\zeta_2\phi + \psi] =0, \quad \hbox{in }
\Omega_{\A}. \label{System2}
\end{equation}

\medskip
As for the boundary conditions we compute
\begin{equation*}
\beta_{\eta}\frac{\pp u_2}{\pp n_{\A}} +  \zeta_{2}\frac{\pp \phi}{\pp n_{\A}} +(u_2 -  \mathbb{H}(x))
\frac{\pp\beta_{\eta}}{\pp n_{\A}} + \phi \frac{\pp \zeta_2}{\pp n_{\A}} + \frac{\pp \psi}{\pp n_{\A}}=0.
\end{equation*}

\medskip
Therefore, as we proceeded above, we reduce the boundary condition to the boundary condition system
\begin{equation}\label{BDRCONDPHI1}
\beta_{\eta}\frac{\pp u_2}{\pp n_{\A}} +  \zeta_{2}\frac{\pp \phi}{\pp n_{\A}}=0
\end{equation}
\begin{equation}\label{BDRCONDPSI1}
\frac{\pp \psi}{\pp n_{\A}}+ (u_2 -  \mathbb{H}(x))
\frac{\pp\beta_{\eta}}{\pp n_{\A}} + \phi \frac{\pp \zeta_2}{\pp n_{\A}}=0.
\end{equation}

 \medskip

Next, we extend \eqref{System1} to a qualitative similar equation in $M_{\A}\times\R$. Let us set
$$
R(\phi) := \zeta_4[\Delta_{\NN_{\A,h}} - \pp_{tt} - \Delta_{M_{\A}}].
$$

Observe that $R(\phi)$ is
understood to be zero for $|t + h(\A y)| > \frac{\eta}{\A} + 2$ and so we consider the equation
$$
\pp_{tt} \phi + \Delta_{M_{\A}} \phi +f'(w(t))\phi =\, -\,
\widetilde{S}(u_2) - R(\phi)
$$

\begin{equation}
\,-\, (f'(u_2) - f'(w(t)))\phi \,-\, \zeta_2 u_2\psi \,-\,
\zeta_2N(\phi + \psi), \quad \hbox{in } M_{\A}\times \R.
\label{System1'}
\end{equation}
where from expression \eqref{BDRERROREXPLICIT} and ommiting the depedence on $k$, we have on every the $k-th$ end $M_{k,\A}$ explicitely
$$
\tilde{S}(u_2) = \tilde{S}(u_1) -\A^2\left\{\Delta_{M}h + |A_{M}|^2h\right\}\pp_{t}\phi_{2} - 2\A \frac{I(\A \theta)}{l_1(\A \theta)l_2(\A \theta)}h\pp_{st}\phi_{2} +2\A^2\frac{I(\A \theta)}{l_1(\A \theta)l_2(\A \theta)}(t+h)\pp_{\rho}h\pp_{tt}\phi_{2}
$$

$$
+2\A\pp_{\rho}\beta(\A s)\pp_{s}\phi_3 + \A^2\pp_{\rho\rho}\beta(\A s)(\phi_2+ \phi_3)-\A l_1^{-2}(\A \theta)\pp_{\vv\vv}h^2\pp_{\theta t}\phi_2
-\A l_2^{-2}(\A \theta)\pp_{\rho\rho}h\pp_{s t}\phi_2 + \A^2 l_1^{-2}(\A \theta)|\pp_{\vv}h|^2\pp_{tt}\phi_2
$$

$$
+ \A^2 l_2^{-2}(\A \theta)|\pp_{\rho\rho}h|^2\pp_{t t}\phi_2-2\A^3 sl_2^{-4}(\A \theta) A(\A \theta)|\pp_{\rho}h|^2
\pp_{tt}\phi_2
+ f'(w(t))(\phi_2\cdot\phi_3 + \frac{1}{2}\phi_3^2)+ [f'(u_1)-f'(w(t))](\phi_2 + \phi_3)
$$

$$
+ \A^2\zeta_{4}\tilde{a}_1(\A s,\A \theta, \A(t+h))\,\left\{\pp_{\theta t}\phi_2 - \A \pp_{\vv}h \phi_{tt}\phi_2\right\}  + \A^2\zeta_{4}\tilde{a}_2(\A s,\A \theta, \A(t+h))\,\left\{\pp_{s t}\phi_2 - \A \pp_{\rho}h \phi_{tt}\phi_2\right\}
$$

$$
+ \A^2\zeta_{4}\tilde{b}_1(\A s,\A \theta, \A(t+h))\,\left\{\pp_{\theta }\phi_2 - \A \pp_{\vv}h \phi_{t}\phi_2\right\}  + \A^2\zeta_{4}\tilde{b}_2(\A s,\A \theta, \A(t+h))\,\left\{\pp_{s}\phi_2 - \A \pp_{\rho}h \phi_{t}\phi_2\right\}
$$

\begin{equation}\label{BDRERROREXPLICIT1} +\A^3\zeta_{4}(t+h)\tilde{b}_3(\A s,\A \theta,\A(t+h))\pp_{t}\phi_2 +\zeta_{4}\tilde{R}_{2,\A}
\end{equation}
and from expression \eqref{ErrorU11} we write
$$
\tilde{S}(u_1) = -\A^2\{\Delta_{M}h + |A_{M}|^2h\}w'(t)
+ \A^2\pp_{i}h\pp_{j}h w''(t)
$$

$$
\A^4\Delta_{M}(|A_{M}|^2)\psi_1(t) - \A^4\{\Delta_{M}h + |A_{M}|^2h\}|A_{M}|^2\psi(t) + \A^4|A_{M}|^4t\pp_{t}\psi_1(t)
$$

$$
- 2\A^4a_{ij}^0(\A y)\pp_{i}h\pp_{j}(|A_{M}|^2)\pp_{t}\psi_1(t) +\A^4a_{ij}^0(\A y)\pp_{i}h\pp_{j}h|A_{M}|^2\pp_{tt}\psi_1(t)
$$

$$
-\A^3\zeta_4(t+h)a_{ij}^1(\A y,\A (t+h))\left(\pp_{ij}h w'(t) - \pp_{i}h\pp_{j}h w''(t)\right)
$$

$$- \A^3\zeta_4b_{i}^1(\A y,\A (t+h))\pp_{i}h w'(t) - \A^4(t+h)\zeta_4b_{3}^1(\A y,\A(t+h))w'(t)
$$

\begin{equation}\label{ErrorU11TILDE}
N(\phi_1)+ \A^5\zeta_4 R_{1,\A}(\A y,t,h,\nabla_{M}h,D^2_{M}h)
\end{equation}

\medskip

Observe that $\tilde{S}(u_1)$ and $\tilde{S}(u_2)$ coincide with $S(u_1)$, $S(u_2)$ but the parts that are not defined for all $t \in \R$ are cut-off outside the support of $\zeta_2$.

\medskip
We proceed in the same fashion for the boundary condition
we and writing
$$
\mathcal{B} = \zeta_4\left[\sqrt{\tilde{g}^{11}}\frac{\pp}{\pp n_{\A}}- \pp_{s}\right]
$$
it suffices to consider $\phi$ satisfying
$$
\pp_{\tau_{\A}}\phi + \mathcal{B}(\phi) = \tilde{B}(u_2)
$$
where $\tau_{\A}=s$  is the tangent inward direction to $\pp M_{\A}$ and in expression \eqref{bdrerror0} we cut-off the parts that are not defined for every $t$. We write also for further purposes
$$
\tilde{B}(u_2)= - \A^3\pp_{\rho}(|A_{M}|^2)\psi_1(t) + \A^3 I(\A \theta)|A_{M}|^2\pp_{t}\psi_1(t)
$$

$$
+\A\left\{\pp_{\rho}h + I(\A \theta)h\right\}\pp_{t}\phi_2 -2m_1(\A \theta)h\pp_{s}\phi_2 - \A^2(t+h)^2\zeta_4\tilde{d}_1(\A \theta)\pp_{s}\phi_2
$$

\begin{equation}\label{bdrerror1}
-2\A^2 I(\A \theta)m_1(\A \theta)hw'(t) - \A^2I(\A \theta)m_1(\A \theta)t^2\pp_{t}\phi_2 +  \A^4\zeta_4\tilde{B}_{0,\A}.
\end{equation}

\medskip
Observe that again we have ommitted the depedence on the end $M_{k,\A}$ for notational convenience.

\medskip
We solve first \eqref{System2}-\eqref{BDRCONDPSI1}, using the fact that the potential $2- (1 -
\zeta_2)[f'(U) +2]$ is uniformly positive, so that the linear operator
behaves like $\Delta - 2$. A solution $\psi = \Psi(\phi)$ is
then found from the contraction mapping principle. We collect this
discussion in the following lemma, that will be proven in detail in
section 5.

\begin{prop}\label{Solving System2}
Assume and $3<p \leq \infty$ and let $h$ be as in \eqref{APRIORIESTNODALSET}. Then, for every $\A
>0$ sufficiently small and every $\phi$ such that $\|\phi\|_{2,p,\sigma}
\leq 1$, equation \eqref{System2} has a unique solution $\psi = \Psi(\phi)$. Even more the operator $\Psi(\phi)$ turns out to be
lipschitz in $\phi$. More precisely, $\Psi(\phi)$ satisfies that
\begin{equation}\label{size of Psi}
\|\psi\|_{X} := \|D^2\psi\|_{p,\ts} + \|D\psi\|_{\infty} +
\|\psi\|_{\infty} \leq Ce^{-\frac{c\eta}{\A}}
\end{equation}
and
\begin{equation} \label{Lipschitz psi}
\|\Psi(\phi_1) - \Psi(\phi_2)\|_{X} \leq
Ce^{-\frac{c\eta}{\A}}\|\phi_1 - \phi_2\|_{2,p,\sigma}.
\end{equation}
\end{prop}

\medskip
Hence, using Proposition \ref{Solving System2}, we solve equation
\eqref{System1'} with $\psi=\Psi(\phi)$. Let us set
$$
\N(\phi) \,:=\,R(\phi)\,+\,  (f'(u_2) - f'(w(t))\phi \,+\,
\zeta_2 (u_2 -  \mathbb{H}(x))\Psi(\phi) \,+\, \zeta_2N(\phi + \Psi(\phi)), \quad
\hbox{in } M_{\A}\times\R.
$$

So, we only need to solve
\begin{equation}\label{System1''}
\pp_{tt} \phi + \Delta_{M_{\A}} \phi +f'(w(t))\phi =\, -\,
\widetilde{S}(u_2) - \N(\phi) + c(y)w'(t), \quad \hbox{in } M_{\A}\times \R.
\end{equation}
\begin{equation}\label{BDRCONDPHI2}
\pp_{\tau_{\A}}\phi + \mathcal{B}(\phi) = \tilde{B}(u_2), \quad \hbox{on } \pp M_{\A}\times \R
\end{equation}
\begin{equation}\label{ORTHOGCOND}
\int_{\R}\phi(\cdot,t)w'(t)dt = 0, \quad y\in M_{\A}
\end{equation}

\medskip
To solve problem \eqref{System1''}-\eqref{BDRCONDPHI2}-\eqref{ORTHOGCOND}, we solve a nonlinear problem in
$\phi$, that basically eliminates the parts of the error, that do not contribute to the projections.

\medskip
The linear theory we develop to solve problem \eqref{System1''}-\eqref{BDRCONDPHI2}-\eqref{ORTHOGCOND}, considers right hand sides and boundary data with a behavior similar to the that of the error $\widetilde{S}(u_2)$ and $\tilde{B}(u_2)$, that as we have seen, is
basically of the form $\OO(e^{-\sigma|t|})$.

\medskip

Using the fact
that $\N(\phi)$ is Lipschitz with small Lipschitz constant and
contraction mapping principle in a ball of radius $\OO(\A^3)$ in the
norm $\|\cdot\|_{2,p,\sigma}$, we solve equation \eqref{System1''}-\eqref{BDRCONDPHI2}-\eqref{ORTHOGCOND}. This
solution $\phi$, defines a Lipschitz operator $\phi=\Phi(h)$. This information is collected in the following proposition

\begin{prop}\label{Solvinng Projected system1''}
Assume $3 < p \leq \infty$ and $\sigma >0$ is small
enough. There exists an universal constant $C >0$, such that
\eqref{System1''}-\eqref{BDRCONDPHI2}-\eqref{ORTHOGCOND} has a unique solution $\phi =
\Phi(h)$, satisfying

$$
\|\phi\|_{2,p,\sigma} \leq C\A^3
$$
and
$$
\|\Phi(h_1) - \Phi(h_2)\|_{2,p,\sigma} \leq C\A^2
\|h_1 - h_2\|_{*}.
$$
\end{prop}

\medskip

\subsection{Adjusting $h$, to make the projection equal
zero.} In this part we set $c_0=\|w'\|^{2}_{L^2(\R)}$. To conclude the proof of Theorem \ref{theo1}, we adjust $h$ so that
$$
c(y)=\int_{\R} \left[\widetilde{S}(u_2) + \N(\phi)\right]w'(t)dt=0.
$$

Let us integrate \eqref{System1''}-\eqref{BDRCONDPHI2}-\eqref{ORTHOGCOND} against $w'(t)$ to find that
$$
\int_{\R}\tilde{S}(u_2)w'(t)dt \,=\, \underbrace{(1 -\beta_{\A})\int_{\R}\tilde{S}(u_1)w'(t)dt}_{A} +\underbrace{\int_{\R}\tilde{S}(u_2)w'(t)dt}_{B} + \OO_{L^{\infty}(M_{\A})}(\A^4)
$$
clearly $\beta_{\A}$ does not depend on $t$ so that we can compute from \eqref{ErrorU11TILDE}
$$
\int_{\R}\tilde{S}(u_1)w'(t)dt = -\A^2\{\Delta_{M}h + |A_{M}|^2h\}c_0
$$

$$
-\A^3\int_{\R}\zeta)4(t+h)a_{ij}^1(\A y,\A(\A(t+h)))\left\{\pp_{ij}hw'(t)
- \pp_{i}h\pp_{j}hw''(t)\right\}w'(t)dt
$$

$$
 - \A^3\int_{\R}\zeta_4(t+h)b_1^1(\A y,\A(t+h))\pp_{i}h(w'(t))^2dt + \A^4\int_{\R}(t+h)^3\zeta_4b_3^1(\A y,\A(t+h))(w'(t))^2dt
$$

$$
+\A^4|A_{M}|^2\int_{\R}t\pp_{t}\psi_1(t)w'(t)dt + \A^5P_{1}(\A y,h,\nabla h,D^2 h).
$$

Next we write the reduced error near the boundary. From \eqref{BDRERROREXPLICIT1} we obtain that
$$
\int_{\R}\tilde{S}(u_2)= \int_{\R}\tilde{S}(u_1) - \A^2|A_{M}|^2 \int_{\R}t\pp_{t}\phi_2w'(t)dt - 2\A\frac{I(\A \theta)}{l_1(\A \theta)l_2(\A \theta)}h\int_{\R}\pp{st}\phi_2w'(t)dt
$$

$$
+2\A^2\frac{I(\A \theta)}{l_1(\A \theta)l_2(\A \theta)}\pp{\rho}\int_{\R}(t+h)\pp_{tt}\phi_2w'(t)dt - \A l_1^{-2}(\A \theta)\pp_{\vv\vv}h\int_{\R}\pp_{\theta t}\phi_2w'(t)dt  - \A l_2^{-2}(\A \theta)\pp_{\rho\rho}h\int_{\R}\pp_{s t}\phi_2w'(t)dt
$$

$$
+\A^2\int_{\R}\zeta_4 \tilde{a}_{1}(\A s,\A \theta,t)\left\{\pp_{\theta t}\phi_{2} - \A \pp_{\vv}\pp_{tt}\phi_2\right\}w'(t)dt +\A^2\int_{\R}\zeta_4 \tilde{a}_{2}(\A s,\A \theta,t)\left\{\pp_{s t}\phi_{2} - \A \pp_{\rho}\pp_{tt}\phi_2\right\}w'(t)dt
$$

$$
+\A^2\int_{\R}\zeta_4 \tilde{b}_{1}(\A s,\A \theta,t)\left\{\pp_{\theta }\phi_{2} - \A \pp_{\vv}\pp_{t}\phi_2\right\}w'(t)dt
+\A^2\int_{\R}\zeta_4 \tilde{b}_{2}(\A s,\A \theta,t)\left\{\pp_{s}\phi_{2} - \A \pp_{\rho}\pp_{t}\phi_2\right\}w'(t)dt
$$

$$
\A^3\int_{\R}\zeta_4(t+h)\tilde{b}_3^1(\A s,\A \theta,\A (t+h))\pp_{t}\phi_2w'(t)dt - 2\A^2\frac{I(\A \theta)}{l_1(\A \theta)l_2(\A \theta)}(t+h)\int_{\R}\left\{\pp_{st}\phi_3 - \A\pp_{\rho\rho}\pp_{tt}\phi_3\right\}w'(t)dt
$$

$$
\int_{\R}[f'(u_1) -f'(w(t))](\phi_2+ \phi_3)w'(t)dt + \A^3\tilde{R}_{0,\A}(\A s,\A \theta) + \A^4\int_{\R}\tilde{R}_{\A}(\A s, \A \theta ,h,\nabla h, D^2 h)w'(t)dt.
$$

\medskip
Also observe that from the nonlocal terms and condition \eqref{APRIORIESTNODALSET} we have
$$
\Q(\A y, h, \nabla h,D^2 h)=\int_{\R}\N(\phi)w'(t)dt , \quad \|\Q(\cdot, h, \nabla h ,D^2 h)\|_{L^{\M}} \geq C\A^{4-\frac{2}{p}}
$$

Which implies that
\begin{equation}\label{REDUCEDEQN}
\A^{-2}\int_{\R}\left(\tilde{S}(u_2) + \N(\phi)\right)w'(t)dt= -c_0\{\Delta_{M}h + |A_{M}|^2h\} + \A P_{0}(y,h,\nabla h,D^2 h) + \A^{2-\frac{2}{p}} P_1(y,h,\nabla h, D^2 h)
\end{equation}
where
$$
\|P_1\|+ \|P_0\|\leq C.
$$
$$
|D P_0| + |D P_1|\leq C.
$$

\medskip
As for the boundary condition, directly from \eqref{bdrerror1} we ask $h$ to satisfy
\begin{equation}\label{ROBINBDRCONDh}
\pp_{\rho}h +I(\A \theta)h=\A c_1 I(\vv)m_1(\vv)+
2m_1(\vv)c_2(\frac{\vv}{\A})
\end{equation}
where we recall that
$$
\|c_{2}\|_{L^{\infty}(\pp M_{\A})}\leq C\A
$$
and the right hand side in \eqref{ROBINBDRCONDh} does not depend on $h$.

\medskip
We solve then
\begin{equation}\label{REDUCEDEQN}
c_0\{\Delta_{M}h + |A_{M}|^2h\} = \A P_{0}(y,h,\nabla h,D^2 h) + \A^{2-\frac{2}{p}} P_1(y,h,\nabla h, D^2 h)
\end{equation}
with the boundary condition \eqref{ROBINBDRCONDh} as a direct consequence of the theory developed in section \ref{JacobiOprExampl} and a fixed point argument for $h$ in a ball or order $\OO(\A)$ in the topology induced by the norm$\|\cdot\|_{*}$. This completes the proof of our theorem.


\medskip
\section{projected linear problem}\label{projectedequation}
In this part we provide the linear theory for the problem
\begin{equation}\label{Projected EQN}
\pp_{tt} \phi + \Delta_{M_{\A}} \phi + f(w(t))\phi = g + c(y)w'(t), \quad \hbox{in }M_{\A} \times \R
\end{equation}
\begin{equation}\label{bdrcondit+orthogcond}
\int_{\R}\phi w'dt=\Lambda(y), \quad  y\in M_{\A}, \quad \frac{\pp \phi}{\pp \tau_{\A}} =h, \quad \hbox{on }\pp M_{\A}\times \R
\end{equation}
relies strongly on the fact that solutions to
$$
\pp_{tt} \phi + \Delta_{M_{\A}} \phi + f(w(t))\phi =0, \quad \hbox{in }M_{\A} \times \R
$$
$$
\frac{\pp \phi}{\pp \tau_{\A}} =0, \quad \hbox{on }\pp M_{\A}\times \R.
$$
are the scalar multiples of $w'(t)$. The proof follos the same lines of lemma 5.1 in \cite{9}. We simply remark that when decomposing the solution $\phi$ as
$$
\phi= c(y)w'(t) + \phi^{\perp}
$$
from maximum principle one obtains that $|\phi^{\perp}(y,t)|\leq C e^{-\sigma|t|}$ for somoe $0<\sigma < \sqrt{2}$. Defining
$$
\psi(y) = \int_{R}|\phi(y,t)^{\perp}|^2 dt
$$
it follows that for certain positive constant $\la$
$$
-\Delta_{M_{\A}}\psi + \la \psi \leq 0, \quad \frac{\ \pp\psi}{\pp \tau_{\A}}=0
$$
where $\tau_{\A}$ is the iward unit tangent to $\pp M_{\A}$ in $M_{\A}$. Clearly it follows that $\psi =0$ and consequently $c(y)$ is a constant function.

\medskip
Proceeding as in section 3 in \cite{delPinoKowalczykWei} it suffices to solve the case $\Lambda= h=0$ and $\int_{g}w'(t)dt=0$. To prove existence we set
$$
<\phi,\psi>:= \int_{M_{\A}\times\R}\nabla \phi \cdot \nabla \psi + 2 \phi\cdot\psi
$$
and we consider the space $H$ of function $\phi\in H^1(M_{\A} \times \R)$ such that
$$
\int_{M_{\A}\times \R}
\phi\cdot w'=0.
$$

\medskip
Since $f'(w(t)) = -2 + \OO(e^{-\sqrt{2}|t|})$ as $|t| \to \infty$, the equation can be put into the setting
$$
(I + K)\phi = f, \quad \hbox{in }H
$$
where $K:H\to H$ is a compact operator. From fredholm alternative it follows that existence and
$$
\|\phi\|_{L^2(M_{\A}\times \R)}\leq C\|g\|_{L^2(M_{\A}\times\R)}.
$$

\medskip
As for the apriori estimates, we can proceed using a blow up argument following the same lines as in the local elliptic regularity as in \cite{9}. We remark that in our case we also need to consider two limiting blow up situations the case of $\R^2 \times \R$ when taking limit well inside $M_{\A}\times \R$ and the case of the half space  $\R_{+}\times \R^2$ when taking the limit in coordinates close to $\pp M_{\A} \times \R$. The former case is reduced to the case of $\R^2\times \R$ as limiting situation by using an odd reflection respect to the boundary of $\R_{+}\times \R^2$.

\medskip
Thus we have proven the following proposition

\begin{prop}\label{ProjectedProblemInvertTheory}
For every $p<3$ and for every $\A>0$ small enough and given arbitrary functions
$g$ defined in $M_{\A}\times \R$ and $G$ defined in $\pp M_{\A}\times \R$ such that
$$
\|g\|_{p,\sigma} + \|G\|_{p,
\sigma}<\infty
$$
there exists a unique bounded solution $\phi$ to problem \eqref{Projected EQN}-\eqref{bdrcondit+orthogcond} satisfying the apriori estimate
$$
\|D^2 \phi\|_{p,\sigma}+\|D\phi\|_{\infty,\sigma}+\|\phi\|_{\infty,\sigma}\leq C(\|g\|_{p,\sigma} + \|G\|_{p,
\sigma})
$$
where the constant $C$ dependens only on $p>0$.
\end{prop}

\section{gluing reduction and solution to the projected problem.}

In this section, we prove Lemma \ref{Solving System2} and then we
solve the nonlocal projected problem  \eqref{System1''}-\eqref{BDRCONDPHI2}-\eqref{ORTHOGCOND}. The notations we use in this section have been set up in sections \ref{APPROXSOL} and \ref{PROOFTHEO1}.

\medskip
\subsection{Solving the Gluing System.}\label{SOLVINGGLUING} Given a fixed $\phi$ such that
$\|\phi\|_{2,p,\sigma} \leq 1$, we solve problem \eqref{System2} with boundary condition \eqref{BDRCONDPSI1}. To begin with, we observe that there exist constants $a < b$, independent of $\A$,
such that
$$
0<a \leq Q_{\A}(x) \leq b, \quad \hbox{ for every } x \in \R^3
$$
where $Q_{\A}(x) = 2 - (1 - \zeta_2)[f'(U)+2]$. Using this remark we study the problem
\begin{equation}
\begin{array}{ccc}
\Delta \psi - Q_{\A}(x) \psi = g(x), \quad x\in\Omega_{\A}
\\
\\
\frac{\pp \psi}{\pp n_{\A}} =G(x), \quad \hbox{on }\pp \Omega_{\A}
\end{array}
\label{GluingEquation}
\end{equation}
for  given $g, G$.  COncerning solvability of this linear problem we have the following lemma.

\medskip
\begin{lemma}
Assume $3<p\leq\infty$ . There exists a constant
$C>0$ and $\A_0 >0$ small enough such that for $0<\A<\A_0$ and any
given $g,G$ with
$$
\|g\|_{L^p(\Omega_{\A})}+ \|G\|_{L^p(\pp \Omega_{\A})} <\infty
$$
equation
\eqref{GluingEquation} has a unique solution $\psi = \psi(g)$,
satisfying the a-priori estimate
$$
\|\psi\|_{X} \leq C(
\|g\|_{L^p(\Omega_{\A})}+ \|G\|_{L^p(\pp \Omega_{\A})})
$$
\end{lemma}

The proof of this lemma is standard and we refer the reader to section 2 in \cite{delPinoKowalczykWei} for details.

Now we prove Proposition \ref{Solving System2}. Denote by $X$, the
space of functions $\psi \in W^{2,p}(\Omega_{\A})$ such that $\|\psi\|_X
<\infty$ and let us denote by $\Gamma(g,G) = \psi$ the solution to the
equation \eqref{GluingEquation}, from the previuos lemma. We see that the linear map $\Gamma$ is continuous i.e
$$
\|\Gamma(g,G)\|_X \leq C(
\|g\|_{L^p(\Omega_{\A})}+ \|G\|_{L^p(\pp \Omega_{\A})}
)
$$
Using this we can recast \eqref{System2} as a fixed point problem,
in the following manner
\begin{equation}
\psi =- \Gamma \left((1-\zeta_2) S(U) + (1- \zeta_2)N[\zeta_2\phi +
\psi], (u_2 -  \mathbb{H}(x))
\frac{\pp\beta_{\eta}}{\pp n_{\A}} + \phi \frac{\pp \zeta_2}{\pp n_{\A}}\right) \label{FixedPointSystem2}
\end{equation}
Let us take $\phi$, $h$satisfying that
$$
\|\phi\|_{2,p,\sigma} \leq 1, \quad \|h\|_{*}\leq K\A.
$$

Estimating the size of the right-hand side in
\eqref{FixedPointSystem2}.

\medskip
Recall that $S(U) = \zeta_2 \tilde{S}(u_2) + E$. So, we estimate directly to
get
$$
|(1 - \zeta_2)\tilde{S}(u_2)|\leq C\A^2e^{-\sigma|t|}(1 -
\zeta_2) \leq C\A^2 e^{-\sigma\frac{\eta}{\A}}
$$
this means that
$$
|(1 - \zeta_2)\tilde{S}(u_2)|\leq C\A^2 e^{-\sigma\frac{\eta}{\A}}
$$
and so $\|(1 - \zeta_2)\tilde{S}(U)\|_{L^{p}(\Omega_{\A})} \leq C\A^2
e^{-\sigma\frac{\eta}{\A}}$.

\medskip
As for the second term in the right-hand side of
\eqref{FixedPointSystem2}, the following holds true
$$
\begin{array}{ccc}
|2\nabla \zeta_2 \cdot \nabla \phi + \phi \Delta \zeta_2| &\leq& C
(1-\zeta_2)e^{-\sigma|t|}\|\phi\|_{2,p\sigma}\\
\\
&\leq& Ce^{-\sigma\frac{\eta}{\A}} \|\phi\|_{2,p,\sigma}.
\end{array}
$$
This implies that
$$
\|2\nabla \zeta_2 \cdot \nabla \phi + \phi \Delta \zeta_2\|_{\infty}
\leq Ce^{-c\frac{\eta}{\A}}.
$$

\medskip
Proceeding in the same fashion, we clearly obtain that the boundary condition satisfies
$$
\left\|(u_2 -  \mathbb{H}(x))
\frac{\pp\beta_{\eta}}{\pp n_{\A}} + \phi \frac{\pp \zeta_2}{\pp n_{\A}}\right\|_{\infty}\leq C e^{-\frac{\sigma\eta}{\A}}
$$

Finally we must check the lipschitz character of $(1-
\zeta_2)N[\zeta_2\phi + \psi]$. Take $\psi_1, \psi_2 \in X$. Then
\begin{eqnarray*}\
|(1- \zeta_2)N[\zeta_2\phi + \psi_1] &-& (1- \zeta_2)N[\zeta_2\phi + \psi_2]| \leq \\
&\leq&(1- \zeta_2)|f(U + \zeta_2 \phi + \psi_1) \\
&&- f(U + \zeta_2 \phi + \psi_2) - f'(U)(\psi_1 - \psi_2)|\\
&\leq& Ce^{-\sigma\frac{\eta}{\A}}(1 - \zeta_2)\sup_{t\in
[0,1]}|\zeta_1\phi + t\psi_1
+(1 -t)\psi_2||\psi_1 - \psi_2|\\
&\leq& Ce^{-\sigma\frac{\eta}{\A}}(\|\phi\|_{\infty,\sigma} +
\|\psi_1\|_{\infty} + \|\psi_2\|_{\infty})|\psi_1 - \psi_2|
\end{eqnarray*}
So, we see that
$$
\|(1- \zeta_2)N[\zeta_2\phi + \psi_1] - (1- \zeta_2)N[\zeta_2\phi +
\psi_1]\|_{\infty} \leq C e^{-\sigma\frac{\eta}{\A}}\|\psi_1 - \psi_2
\|_{\infty}
$$
In particular we see that $\|(1-
\zeta_2)N(\zeta_2\phi)\|_{\infty} \leq C e^{-\sigma\frac{\eta}{\A}}$.
Consider $\widetilde{\Gamma}:X \to X$,
$\widetilde{\Gamma}=\widetilde{\Gamma}(\psi)$ the operator given by
the right-hand side of \eqref{FixedPointSystem2}. From the previous
remarks we have that $\widetilde{\Gamma}$ is a contraction provided
$\A$ is small enough and so we have found $\psi =
\widetilde{\Gamma}(\psi)$ the solution to \eqref{System2}.

\medskip
We can check directly that $\Psi(\phi) = \psi$ is lipschitz in
$\phi$ i.e

\begin{eqnarray*}
\|\Psi(\phi_1) - \Psi(\phi_2)\|_X \leq C \|(1
-\zeta_2)[N(\zeta_1\phi_1 + \Psi(\phi_1)) -
N(\zeta_1\phi_2 + \Psi(\phi_2))]\|_{\infty,\mu}\\
+ C^{-\sigma\frac{\sigma \eta}{\A}}\|\phi_1 - \phi_2\|_{2,p,\sigma}\\
\leq Ce^{-c\frac{\eta}{\A}}\left(\|\Psi(\phi_1) - \Psi(\phi_2)\|_X +
\|\phi_1 - \phi_2\|_{2,p,\sigma}\right)
\end{eqnarray*}
Hence for $\A$ small, we conclude
$$
\|\Psi(\phi_1) - \Psi(\phi_2)\|_X \leq C e^{-c\frac{\eta}{\A}}
\|\phi_1 - \phi_2\|_{2,p,\sigma} .
$$

\medskip

\subsection{Solving the Projected Problem.} Now we solve problem
\eqref{System1''}-\eqref{BDRCONDPHI2}-\eqref{ORTHOGCOND} using the linear theory developed in
section \ref{projectedequation}, together with a fixed point argument. From the discussion
in \ref{SOLVINGGLUING}, we have a nonlocal operator $\psi = \Psi(\phi)$.

\medskip
Recall that
$$
\N(\phi) \,:=\,R(\phi)\,+\,  (f'(u_2) - f'(w(t))\phi \,+\,
\zeta_2 (u_2 -  \mathbb{H}(x))\Psi(\phi) \,+\, \zeta_2N(\phi + \Psi(\phi)), \quad
\hbox{in } M_{\A}\times\R.
$$

Let us denote
$$
N_1(\phi):= R(\phi) + \left[f'(u_2) - f'(w(t))\right]\phi
$$

$$
N_2(\phi) := \zeta_2(u_2 -  \mathbb{H}(x))\Psi(\phi)
$$

$$
N_3(\phi) := \zeta_2N(\phi + \Psi(\phi))
$$
We need to investigate the Lipschitz character of $N_i$, $i=1,2,3$.
We see that
$$
|N_3(\phi_1) - N_3(\phi_2)| = \zeta_2|N(\phi_1 + \Psi(\phi_1)) -
N(\phi_2 + \Psi(\phi_2))|
$$
$$
\leq C \zeta_2 \sup_{\tau \in [0,1]} |\tau(\phi_1 + \Psi(\phi_1)) +
(1-\tau)(\phi_2 + \Psi(\phi_2))|\cdot |\phi_1 - \phi_2 +
\Psi(\phi_1) - \Psi(\phi_2)|
$$
$$
\leq C\left[|\Psi(\phi_2)| + |\phi_1  -\phi_2| + |\Psi(\phi_1) -
\Psi(\phi_2)|+|\phi_2| \right]\cdot \left[|\phi_1 - \phi_2| +
|\Psi(\phi_1) - \Psi(\phi_2)|\right].
$$
This implies that
$$
\|N_3(\phi_1) - N_3(\phi_2)\|_{p,\sigma} \leq
C[e^{-\sigma\frac{\eta}{\A}} + \|\phi_1\|_{p,\sigma} +
\|\phi_2\|_{p,\sigma}] \cdot\|\phi_1 - \phi_2\|_{p,\sigma}.
$$

Now we check on $N_1(\phi)$. Clearly, we just have to pay attention
to $R(\phi)$. But notice that $R(\phi)$ is linear on $\phi$ and
$$
R(\phi)= - \A^2 \left\{ h''(\A y) +   \frac {\A y}{ 1+ (\A y)^2 }
h'(\A y) + \frac {2 (t+ h) }{(1+ (\A y)^2)^2} \right\}\pp_t \phi
$$
$$
-2\A h'(\A y)\pp_{ty}\phi + \A^2 [h'(\A y)]^2\pp_{tt}\phi
 + D_{\A,h}(\phi).
$$

Hence, from the assumptions made on $h$, we have that
$$
\|N_1(\phi_1) - N_1(\phi_2)\|_{p,\sigma} \leq C\A \|\phi_1 -
\phi_2\|_{2,p,\sigma}.
$$

\medskip
Observe also that under the assumption made on $h$ we have
$$
\|\tilde{S}(u_2)+\A^2\{\Delta_{M} h + |A_{M}|^2 h\}w'(t)\|_{p,\sigma} \leq C \A^3
$$
Hence we for $\|\phi\|_{2,p,\sigma}\leq A\A^2$ we have that
$\|N(\phi)\|_{p,\sigma} \leq C \A^4$.

\medskip
As for the boundary condition we check directly from  expressions \eqref{bdrcondition}, \eqref{BDRCONDPHI2} and \eqref{BDRCONDPHI2} thta on every end $M_{k,\A}$ the following estimates hold
$$
\|\tilde{B}(u_2)\|_{\infty,\sigma}\leq C\A^3, \quad \|\mathcal{B}(\phi)\|_{\infty,\sigma}\leq C\A(\|\nabla \phi\|_{\infty,\sigma}+\|\phi\|_{\infty,\sigma})
$$
with $\mathcal{B}(\phi)$ linear in $\phi$.

\medskip
Setting $T(g,G) = \phi$ the linear operator given from proposition \ref{ProjectedProblemInvertTheory}, we recast problem
\eqref{System1''}-\eqref{BDRCONDPHI2}-\eqref{ORTHOGCOND} as the fixed point problem
$$
\phi=T(-\tilde{S}(u_2) - \N(\phi), \tilde{B}(u_2) - \mathcal{B}(\phi)) =: \T(\phi)
$$

in the ball
$$
B_{\A}^{X}:=\left\{\phi \in X / \|\phi\|_{2,p,\sigma}\leq
A\A^3\right\}
$$
where $X$ is the space of function $\phi\in W_{loc}^{2,p}(\M_{\A}\times \R)$
with the norm $\|\phi\|_{2,p,\sigma}$. Observe that
$$
\|\T(\phi_1) - \T(\phi_2)\|_X \leq C\|\N(\phi_1) -
\N(\phi_2)\|_{p,\sigma} \leq C\A \|\phi_1 - \phi_2\|_{X}, \quad \phi
\in B_{\A}^{X}.
$$
On the other hand, because $C$ and $A$ are universal constants and
taking $A$ large enough, we have that
$$
\|\T(\phi)\|_X \leq C (\|\tilde{S}(u_2)\|_{p,\sigma} +
\|\N(\phi)\|_{p,\sigma} + \|\tilde{B}(u_2)\|_{\infty,\sigma} + \|\mathcal{B}(\phi)\|_{\infty,\sigma}) \leq A\A^3, \quad \phi \in B_{\A}^{X}.
$$

Hence, the mapping $\T$ is a contraction from the ball $B_{\A}^{X}$
onto itself. From the contraction mapping principle we get a unique
solution $\phi$ as require. We denote the solution to \eqref{System1''}-\eqref{BDRCONDPHI2}-\eqref{ORTHOGCOND}
for $h$ fixed.

As for the Lipschitz character of $\Phi(h)$ it comes from a
lengthy by direct computation. We left to the reader to check on the
details of the proof of the following estimate

$$
\|\Phi(h_1) - \Phi(h_2)\|_{2,p,\sigma} \leq
C\A^2\|h_1 - h_2\|_{*}.
$$

\end{document}